\documentclass[a4paper, 10pt, twoside]{article}

\usepackage{amsmath, amscd, amsfonts, amssymb, amsthm, latexsym, url, color, todonotes, bm, framed, rotating, enumerate} 
\usepackage{graphicx}
\usepackage[left=1in, right=1in, top=1.2in, bottom=1in, includefoot, headheight=13.6pt]{geometry}
\usepackage{mathtools}
\input{xy}
\xyoption{all}

\usepackage[colorlinks,breaklinks=false, bookmarks=false]{hyperref}
\usepackage[figure,table]{hypcap}
\hypersetup{
	bookmarksnumbered,
	pdfstartview={FitH},
	citecolor={black},
	linkcolor={black},
	urlcolor={black},
	pdfpagemode={UseOutlines}
}
\makeatletter
\newcommand\org@hypertarget{}
\let\org@hypertarget\hypertarget
\renewcommand\hypertarget[2]{%
  \Hy@raisedlink{\org@hypertarget{#1}{}}#2%
} 
\makeatother

\newtheorem{theorem}{Theorem}[section]
\newtheorem{lemma}[theorem]{Lemma}
\newtheorem{corollary}[theorem]{Corollary}
\newtheorem{proposition}[theorem]{Proposition}

\theoremstyle{definition}
\newtheorem{definition}[theorem]{Definition}
\newtheorem{remark}[theorem]{Remark}

\newcommand{\xysquare}[8]{
\[\xymatrix{
#1 \ar@{#5}[r] \ar@{#6}[d] & #2 \ar@{#7}[d]\\
#3 \ar@{#8}[r] & #4
}\]
}

\DeclareMathOperator*{\projlimf}{``\varprojlim''}

\newcommand{\al}{\alpha}
\newcommand{\bb}{\mathbb}

\newcommand{\blob}{\bullet}
\newcommand{\bor}{\partial}

\newcommand{\comment}[1]{}

\newcommand{\ep}{\varepsilon}

\newcommand{\into}{\hookrightarrow}

\newcommand{\isoto}{\stackrel{\simeq}{\to}}

\newcommand{\mult}[1]{#1^{\!\times}}

\newcommand{\op}{\operatorname}
\newcommand{\pid}[1]{\langle #1 \rangle}
\renewcommand{\phi}{\varphi}
\newcommand{\quis}{\stackrel{\sim}{\to}}
\newcommand{\res}{\overline}
\newcommand{\roi}{\mathcal{O}}

\newcommand{\sub}[1]{{\mbox{\rm \scriptsize #1}}}

\newcommand{\To}{\longrightarrow}
\newcommand{\ul}[1]{\underline{#1}}

\newcommand{\xto}{\xrightarrow}

\newcommand{\lrho}{{\sigma}}

\renewcommand{\cal}{\mathcal}
\renewcommand{\hat}{\widehat}
\renewcommand{\frak}{\mathfrak}
\newcommand{\indlim}{\varinjlim}
\renewcommand{\tilde}{\widetilde}

\renewcommand{\ker}{\operatorname{Ker}}
\renewcommand{\projlim}{\varprojlim}

\DeclareMathOperator{\dlog}{dlog}

\DeclareMathOperator{\Frac}{Frac}

\DeclareMathOperator{\Spec}{Spec}

\DeclareMathOperator{\Spf}{Spf}

\newcommand{\CH}{C\!H}

\newcommand{\dotimes}{\otimes^{\bb L}}

\usepackage[bbgreekl]{mathbbol}
\DeclareSymbolFontAlphabet{\mathbbm}{bbold}

\newcommand{\cosimp}[3]{\xymatrix@1{#1 \ar@<.4ex>[r] \ar@<-.4ex>[r] & {\ }#2 \ar@<0.8ex>[r] \ar[r] \ar@<-.8ex>[r] & {\ } #3 \ar@<1.2ex>[r] \ar@<.4ex>[r] \ar@<-.4ex>[r] \ar@<-1.2ex>[r] & \cdots }}

\usepackage{fancyhdr}

\usepackage{sectsty}
\sectionfont{\Large\sc\centering}
\chapterfont{\large\sc\centering}
\chaptertitlefont{\LARGE\centering}
\partfont{\centering}

\begin{document}
\itemsep0pt

\title{Milnor $K$-theory of $p$-adic rings}

\author{Morten L\"uders and Matthew Morrow}

\date{}

\maketitle

\begin{abstract}
We study the mod $p^r$ Milnor $K$-groups of $p$-adically complete and $p$-henselian rings, establishing in particular a Nesterenko--Suslin style description in terms of the Milnor range of syntomic cohomology. In the case of smooth schemes over complete discrete valuation rings we  prove the mod $p^r$ Gersten conjecture for Milnor $K$-theory locally in the Nisnevich topology. In characteristic $p$ we show that the Bloch--Kato--Gabber theorem remains true for valuation rings, and for regular formal schemes in a pro sense.
\end{abstract}


\setcounter{section}{-1}
\section{Introduction}
In this article we study the mod $p$-power Milnor $K$-theory of $p$-adic rings by combining classical techniques of Bloch--Kato \cite{Bloch1986} with recent progress in $p$-adic motivic cohomology, notably the syntomic cohomology introduced by Bhatt, Scholze, and the second author \cite{BhattMorrowScholze2}.

Our fundamental results from which the others follow are the following ``Gersten injectivity'' and Bloch--Kato isomorphism:

\begin{theorem}\label{introduction_thm1}
Let $V$ be a complete discrete valuation ring of mixed characteristic, and $R$ a local, $p$-henselian, ind-smooth $V$-algebra with infinite residue field; let $j,r\ge 1$. Then
\begin{enumerate}
\item the canonical map $K_j^M(R)/p^r\to K_j^M(R[\tfrac1p])/p^r$ is injective,
\item and the Galois symbol $K_j^M(R[\tfrac1p])/p^r\to H^j_\sub{\'et}(R[\tfrac1p],\mu_{p^r}^{\otimes j})$ is an isomorphism.
\end{enumerate}
\end{theorem}

The hypothesis that $R$ has infinite residue field may be weakened to it having big enough residue field so that $\hat K_j^M(R)\to K_j^M(R)$ is an isomorphism, where the domain denotes the improved Milnor $K$-theory of Gabber and Kerz \cite{Kerz2010}. However, with no size hypotheses on the residue field of $R$ we prove exactness of the sequence \begin{equation}0\To \hat K_j^M(R)/p^r\To H^j_\sub{\'et}(R[\tfrac1p],\mu_{p^r}^{\otimes j})\To W_r\Omega_{R/\pi R,\sub{log}}^{j-1}\To 0,\label{introduction_eqn}\end{equation} where the map to the group of de Rham--Witt $\dlog$ forms of the special fibre of $R$ is Kato's residue map; in the case of big enough residue field this is equivalent to Theorem \ref{introduction_thm1}.

Combining Theorem \ref{introduction_thm1} with existing Gersten results in Milnor $K$-theory due to Kerz \cite{Kerz2009, Kerz2010} and in motivic cohomology due to Geisser \cite{Geisser2004}, we then establish the mod $p$-power Gersten conjecture Nisnevich locally on smooth $V$-schemes:

\begin{theorem}\label{intro_Gesten}
Let $V$ be a complete discrete valuation ring of mixed characteristic, $X$ a smooth $V$-scheme, and $R:=\roi_{X,x}^h$ the henselian local ring at any point $x\in X$. Then the Gersten complex \[0\To \hat K_j^M(R)/p^r\To K_j^M(\Frac R)/p^r\To \bigoplus_{x\in \Spec R^{(1)}}K_{j-1}^M(k(x))/p^r\To \bigoplus_{x\in \Spec R^{(2)}}K_{j-2}^M(k(x))/p^r\To\cdots\]
is exact, and consequently the Bloch--Quillen isomorphism $\CH^j(X)/p^r\cong H^j_\sub{Nis}(X,\hat{\cal K}^M_{j,X}/p^r)$ holds.
\end{theorem}

A new, albeit simple, observation which plays a key role in our arguments is that Milnor $K$-theory is right exact when left Kan extending from smooth algebras; this is inspired by Bhatt--Lurie's observation that connective $K$-theory is left Kan extended in this fashion, and by the subsequent on-going introduction of left Kan extensions from smooth algebras into the theory of motivic cohomology \cite{ElmantoHoyoisKhanSosniloYakerson2019}. More precisely, given a base ring $k$ and letting $L^\sub{sm}K_j^M:k\op{-algs}^\sub{loc}\to D(\bb Z)$ be the left Kan extension of $K_j^M$ from local, essentially smooth $k$-algebras to all local $k$-algebras, then show in Proposition \ref{proposition_LKE_Milnor} that the counit map $H_0(L^\sub{sm}K_j^M(A))\to K_j^M(A)$ is an isomorphism on any local $k$-algebra $A$.

This left Kan extension observation also provides a new tool to control the Milnor $K$-theory of $\bb F_p$-algebras; in particular, we prove the following Bloch--Kato--Gabber style isomorphisms, which are already known to hold \cite{ClausenMathewMorrow2020, KellyMorrow2018a, Morrow_pro_GL2} if improved Milnor $K$-theory is replaced by algebraic $K$-theory:

\begin{theorem}\label{intro_charp}
Let $r,j\ge0$.
\begin{enumerate}
\item Let $A$ be a local, Cartier smooth $\bb F_p$-algebra (e.g., an essentially smooth, local algebra over a characteristic $p$ valuation ring). Then the dlog map $\hat K_j^M(A)/p^r\to W_r\Omega^j_{A,\sub{log}}$ is an isomorphism.
\item Let $A$ be an F-finite, regular, Noetherian, local $\bb F_p$-algebra, and $I\subseteq A$ any ideal. Then the dlog map induces an isomorphism of pro abelian groups $\{\hat K_j^M(A/I^s)/p^r\}_s\isoto\{W_r\Omega^j_{A/I^s,\sub{log}}\}_s$.
\end{enumerate}
\end{theorem}

To state our final main result and outline the proofs, we briefly recall the syntomic cohomology theory $\bb Z_p(j)(A)$ introduced in \cite{BhattMorrowScholze2} for quasisyntomic rings and extended to arbitrary $p$-complete rings $A$ in \cite{AntieauMathewMorrowNikolaus}. This weight-$j$ syntomic cohomology $\bb Z_p(j)(A)$ is a $p$-complete complex which is expected to provide a general theory of $p$-adic \'etale motivic cohomology for $A$; consequently, by a Beilinson--Lichtenbaum principle, one expects the truncation $\tau^{\le j}\bb Z_p(j)(A)$ to share certain properties with Zariski motivic cohomology. As part of these expectations we prove the following relation to Milnor $K$-theory, thereby providing a $p$-adic analogue of the Nesterenko--Suslin isomorphism $K_j^M(k)\cong H^j_\sub{Mot}(k,\bb Z(j))$ for fields $k$ \cite{Suslin1989}.

\begin{theorem}\label{intro_NS}
For any local, $p$-complete ring $A$, there are natural isomorphisms \[\hat K_j^M(A)/p^r\isoto H^j(\bb Z_p(j)(A)/p^r)\] for all $r,j\ge0$.
\end{theorem}

We finish the introduction by sketching the proofs of the main theorems, ignoring completely the difficulties caused by possible small residue fields.

Theorems \ref{introduction_thm1} and \ref{intro_NS} are proved by a convoluted reduction to a special case which we treat by hand. Let $R$ be as in Theorem \ref{introduction_thm1}. First we use a forthcoming result of Bhatt--Clausen--Mathew identifying the syntomic cohomology $\bb Z_p(j)(R)$ with older approaches to $p$-adic \'etale motivic cohomology (see Theorem \ref{theorem_BCM_LKE} for details), a consequence of which is that the kernel of Kato's residue map in (\ref{introduction_eqn}) identifies with $H^j(\bb Z_p(j)(R)/p^r)$. In other words, for such $R$, Theorems \ref{introduction_thm1} and \ref{intro_NS} are equivalent. Independently, rigidity results of Antieau--Mathew--Nikolaus joint with the second author \cite{AntieauMathewMorrowNikolaus} imply that $\tau^{\le j}\bb Z_p(j)(-)$ is left Kan extended from smooth algebras in a suitable sense (see Proposition \ref{proposition_motivic_cohomol_LKE} for details); combined with the right exactness property of Milnor $K$-theory mentioned above, this reduces Theorem \ref{intro_NS} to the case in which the ring is ind-smooth over $\bb Z_p$. Putting together these two steps (and making a standard type of argument by adding roots of unity to reduce mod $p^r$ assertions to the mod $p$ case), it is finally enough to prove Theorem \ref{introduction_thm1} when $R$ is ind-smooth over $\bb Z_p$ and $r=1$.

The main difficulty is then to show that the the composition $K_j^M(R)/p\to K_j^M(R[\tfrac1p])\to H^j_\sub{\'et}(R[\tfrac1p],\mu_p^{\otimes j})$ is injective. Using Bloch--Kato's filtration on $p$-adic nearby cycles, this reduces when $p\neq 2$ to showing that their differential map $\Omega^{j-1}_{R/pR}\to H^j_\sub{\'et}(R[\tfrac1p],\mu_p^{\otimes j})$ factors through the aforementioned composition; we construct this factoring, which we suspect is known to experts but whose proof does not appear in the literature, using Dennis--Stein symbols. When $p=2$ (but still over base dvr $\bb Z_p$, so that $\tfrac{ep}{p-1}=2$) a further step of the Bloch--Kato filtration, corresponding to differential maps out of the cokernel of the Artin--Schreier maps $C^{-1}-1:\Omega^{j-\ep}_{R/pR}\to\Omega^{j-\ep}_{R/pR}/d\Omega^{j-\ep-1}_{R/pR}$ for $\ep=1,2$, must be analysed. Here the manipulations of the Dennis--Stein symbols are more involved, complicated further by the fact that the restriction of the unit filtration on $K_j^M(R[\tfrac1p])$ to $K_j^M(R)$ had not been previously identified; see \S\ref{subsection_K2}--\ref{subsection_differential _maps}.

\begin{remark}[Unit group filtrations]
Let $V$ be a complete discrete valuation ring of mixed characteristic, with uniformiser $\pi$ and field of fractions $K$; although the observations of this remark continue to hold for suitably smooth local $V$-algebras, notably as in Theorem \ref{introduction_thm1}, we restrict to $V$ itself for the sake of concreteness.

The $i^\sub{th}$ step $U^iK_j^M(K)$ of the {\em unit group filtration} on $K_j^M(K)$ is defined to be the subgroup generated by symbols $\{a_1,\dots,a_j\}$ where at least one of $a_1,\dots,a_j\in K^\times$ belongs to $1+\pi^i V$. The graded pieces of this filtration have been extensively studied, by Bloch--Kato \cite{Bloch1986}, Kurihara \cite{Kurihara1988}, Nakamura \cite{Nakamura2000}, and others; see \cite{Nakamura2000a} for a survey.

Filtering $K_j^M(V)$ in the analogous way (i.e., using $a_1,\dots,a_j\in V^\times$ with at least one belonging to $1+\pi^i V$) results in pathological behaviour: this filtration is not multiplicative and the canonical map $K_j^M(V)\to K_j^M(K)$ is not injective on graded pieces. The general goal of \S\ref{subsection_K2}--\ref{subsection_differential _maps} is to describe what we believe is the ``correct'' filtration on $K_j^M(V)$ (namely the restriction of the unit group filtration on $K_j^M(K)$, assuming injectivity of $K_j^M(V)\to K_j^M(K)$) in the special case that $V$ is absolutely unramified: in fact, the pathological behaviour in this case only occurs when $p=2$, in which case one must enlarge the second step of the filtration by including symbols $\{1+ap,1+bp\}$ where $a,b\in V$, or equivalently by adding certain Dennis--Stein symbols. See Definition \ref{definition_Vfiltration} and Remark \ref{remark_filtration}.

This suggests an alternative approach to our proof of the injectivity of $K_j^M(V)/p^r\to H^j_\sub{\'et}(K,\mu_{p^r}^{\otimes j})$. Firstly enlarge the pathological filtration on $K_j^M(V)$, probably by modifying it in degrees $i$ whenever $i$ is divisible by $p$ and in the range $1\le i\le \tfrac{ep}{p-1}$ (c.f., the Bloch--Kato description of the graded pieces in these degrees \cite[Corol.~1.4.1]{Bloch1986}, noting that the graded step of degree $\tfrac{ep}{p-1}$ might not vanish as we do not work \'etale locally). Secondly, reduce to the case $r=1$ and $V$ containing a primitive $p^\sub{th}$-root of unity. Thirdly, refine the arguments of \S\ref{subsection_differential _maps} to show that Bloch--Kato's differential maps describing the graded pieces of $H^j_\sub{\'et}(K,\mu_p^{\otimes j})$ \cite[Thm.~6.7]{Bloch1986} factor through the graded pieces of $K_j^M(V)/p$.
\end{remark}

We finish the sketches of the main theorems. Theorem \ref{intro_Gesten} follows from Theorem \ref{introduction_thm1} and known Gersten results in motivic and \'etale cohomology, so we refer the reader directly to \S\ref{ss_Gersten} for the details.

Theorem \ref{intro_charp} reduces via Theorem \ref{intro_NS} (which in characteristic $p$ is a consequence of the classical Bloch--Kato--Gabber theorem and our left Kan extension arguments) to describing $H^j(\bb Z(j)(-)/p^r)$ of valuation rings and of regular Noetherian rings modulo powers of an ideal. This in turn reduces to calculations in derived de Rham cohomology, which in the first case are due to Gabber--Ramero \cite[Thm.~6.5.8(ii) \& Corol.~6.5.21]{GabberRamero2003} and Gabber \cite[App.]{KerzStrunkTamme2018}, and in the second case may be found in work of the second author \cite{Morrow_pro_GL2}.

\subsection*{Acknowledgements}
We thank Bhargav Bhatt, Dustin Clausen, and Akhil Mathew for discussions about their work and Shuji Saito for related correspondence.

The first author is supported by the DFG Research Fellowship LU 2418/1-1. The second author was partly supported by the ANR JCJC project {\em P\'eriodes en G\'eom\'etrie Arithm\'etique et Motivique} (ANR-18-CE40-0017).

\section{Formulation of main theorems for \texorpdfstring{$p$}{p}-henselian, ind-smooth algebras}
In this section we state our main theorems concerning ind-smooth algebras over complete discrete valuation rings, and establish some relations between them as well as the key reductions to the special case which will then be proved in Section \ref{section_unramified_p}.

\subsection{Main theorems}\label{subsection_main_thms}
We adopt the usual definition of Milnor $K$-theory, even for non-local rings:

\begin{definition}[Milnor $K$-theory]
Let $R$ be a (always commutative) ring. We define the $j^\sub{th}$ Milnor K-group $K^M_j(R)$ to be the quotient of $(R^{\times})^{\otimes j}$ by the Steinberg relations, i.e. the subgroup of $(R^{\times})^{\otimes j}$ generated by elements of the form $a_1\otimes\cdots\otimes a_j$ where $a_l+a_k=1$ for some $1\leq l<k\leq j$. As usual, the image of $a_1\otimes\cdots\otimes a_j$ in $K^M_j(R)$ is denoted by $\{a_1,\dots,a_j\}$.
\end{definition}

For any ring $R$ and natural number $p$ invertible in $R$, we let \[h_p^j:K_j^M(R)/p\To H^j_\sub{\'et}(R,\mu_{p}^{\otimes j})\] denote Tate's cohomological symbol, also known as the Galois symbol or norm residue homomorphism. For a proof of the existence of the Galois symbol we refer to \cite{Tate1976} where it is proven for $R$ being a field. The proof in the above generality is analogous. 


\begin{remark}[Big residue fields and improved Milnor $K$-theory]\label{remark_improved}
Suppose in this remark that $R$ is a local ring of arbitrary residue characteristic, and consider the restriction of the Galois symbol for $R[\tfrac1p]$ to the Milnor $K$-theory of $R$, namely \[K_j^M(R)\To K_j^M(R[\tfrac1p])\xto{h_p^j} H^j_\sub{\'et}(R[\tfrac1p],\mu_{p}^{\otimes j}).\] This symbol factors through Gabber--Kerz' improved Milnor $K$-theory $\hat K_j^M(R)$ \cite{Kerz2010} (which we recall is a quotient of $K_j^M(R)$ by \cite[Thm.~13]{Kerz2010}). Indeed, letting $R(t)$ and $R(t_1,t_2)$ denote the rational function rings of \cite[Lem.~8]{Kerz2010} (which have infinite residue field), we have \[\hat K_j^M(R)=\op{ker}(K_j^M(R(t))\xto{i_{1*}-i_{2*}}K_j^M(R(t_1,t_2))\] by definition, and \[H^j_\sub{\'et}(R[\tfrac1p],\mu_p^{\otimes j})=\op{ker}(H^j_\sub{\'et}(R(t)[\tfrac1p],\mu_p^{\otimes j})\xto{i_{1*}-i_{2*}}H^j_\sub{\'et}(R(t_1,t_2)[\tfrac1p],\mu_p^{\otimes j}))\] by \cite[Prop.~9]{Kerz2010} since \'etale cohomology has transfers for finite \'etale extensions of rings; naturality of the Galois symbol now proves the claim.

Now let $S$ be a local, finite \'etale $R$-algebra. Then there exist norm maps $N:\hat K_*^M(S)\to \hat K_*^M(R)$ on improved Milnor $K$-theory \cite[\S1]{Kerz2010} (depending, for example, on a chosen presentation of $S(t)$ as a finite \'etale $R(t)$-algebra), and we expect that these are compatible with the transfer maps on \'etale cohomology in the sense that the following diagram commutes
\[\xymatrix{
\hat K_j^M(S)\ar[r]\ar[d]_N & H^j_\sub{\'et}(S[\tfrac1p],\mu_p^{\otimes j})\ar[d]^N\\
\hat K_j^M(R)\ar[r] & H^j_\sub{\'et}(R[\tfrac1p],\mu_p^{\otimes j}).
}\]
We will prove this compatibility in the special case of interest to us in Lemma \ref{lemma_norms_are_comp}. (Here is a possible method of proof in general, which in the case of fields can be found in \cite[pg.~237]{Gille2006}. Since the norm maps on improved Milnor $K$-theory are induced by those on the $K$-theory of $K_j^M(-(t))$ and $K_j^M(-(t_1,t_2))$, this compatibility immediately reduces to the case that $R$ has infinite residue field. Now the norm maps on for Milnor $K$-theory for local rings infinite residue field are defined in \cite[Def. 5.5]{Kerz2010} using an analogue of the Milnor--Bass--Tate sequence for local rings with infinite residue field. Similarly one should be able to prove an analogue of Faddeev's exact sequence for \'etale cohomology for local rings with infinite residue field to define the norm map in that setting. Then one checks the compatibility of the two sequences.)

We recall also that the canonical quotient map $K_j^M(R)\to \hat K_j^M(R)$ is an isomorphism if the residue field of $R$ has $>M_j$ elements \cite[Prop.~10(5)]{Kerz2010}, where $M_j$ is a certain bound independent of $R$ (and implicitly chosen to satisfy $1=M_1\le M_2\le M_3\le\cdots$). Since $j$ will be clear from the context, we will say in this case simply that $R$ has ``big residue field''.
\end{remark}

The following is the first formulation of our main result; by {\em $p$-henselian} we mean that the ring is henselian along the ideal generated by $p$.

\begin{theorem}\label{theorem_BK}
Let $V$ be a complete discrete valuation ring of mixed characteristic, and $R$ a local, $p$-henselian, ind-smooth $V$-algebra; let $j,r\ge 1$ and assume $R$ has big residue field. Then
\begin{enumerate}
\item the cohomological symbol \[h^j_{p^r}:K_j^M(R[\tfrac1p])/p^r\To H^j_\sub{\'et}(R[\tfrac1p],\mu_{p^r}^{\otimes j})\] is an isomorphism;
\item the canonical map \[K_j^M(R)/p^r\To K_j^M(R[\tfrac1p])/p^r\] is injective.
\end{enumerate}
\end{theorem}

\begin{remark}[$j=2$]\label{remark_j=2}
The arguments of Dennis--Stein \cite{Dennis1975} may be used to prove Theorem \ref{theorem_BK}(ii) in the  case $j=2$: take the exact sequence of \cite[Corol.~4.3]{Morrow_K2} (with $t$ being a uniformiser of $V$) mod $p^r$. This special case will actually be used in the course of our proof of the general case (see the end of the proof of Proposition \ref{prop_reduction0}).
\end{remark}

In order to avoid the assumption that $R$ has big residue field, we now reformulate Theorem \ref{theorem_BK} to avoid any reference to the Milnor $K$-theory of $R[\tfrac1p]$. So let $V$ and $R$ be as in Theorem \ref{theorem_BK} but drop the assumption that $R$ has big residue field. Letting $\frak m\subseteq V$ be the maximal ideal (notation which will be used throughout the paper), note that $\frak m R$ is a prime ideal of $R$ and the localisation $R_{\frak m R}$ is a discrete valuation ring (since it is a filtered colimit of dvrs, all of whose maximal ideals are generated by $\frak m$); let $F$ be the field of fractions of its henselisation $R_{\frak m R}^h$, and note that the residue field of $F$ is $\Frac(R/\frak mR)$. Then $F$ is a henselian discrete valuation field of mixed characteristic, and so the cohomological symbol $h_{p^r}^j:K_j^M(F)/p^r\to H^j_\sub{\'et}(F,\mu_{p^r}^{\otimes j})$ is an isomorphism by Bloch--Kato \cite[Thm.~5.12]{Bloch1986}. This may be used to define Kato's residue map \begin{equation}\bor:H^j_\sub{\'et}(R[\tfrac1p],\mu_{p^r}^{\otimes j})\To W_r\Omega_{R/\frak m R,\sub{log}}^{j-1}\label{eqn_Katos_residue}\end{equation} as usual:

\begin{lemma}\label{lemma_Katos_residue}
With notation as in the previous paragraph, the composition \[H^j_\sub{\'et}(R[\tfrac1p],\mu_{p^r}^{\otimes j})\To H^j_\sub{\'et}(F,\mu_{p^r}^{\otimes j})\stackrel{h_{p^r}^j}\cong K_j^M(F)/p^r\xto{\bor} K_{j-1}^M(\Frac(R/\frak m R))/p^r\xto{\dlog} W_r\Omega^{j-1}_{\Frac(R/\frak m R),\sub{log}}\] lands inside $W_r\Omega_{R/\frak m R,\sub{log}}^{j-1}$, where $\bor$ is the boundary map in Milnor $K$-theory for the discrete valuation field $F$.
\end{lemma}
\begin{proof}
Throughout the paper we adopt the following notation when given a $V$-scheme $X$: the inclusion of the special and generic fibres are denoted by $\frak i:X\times_VV/\frak m\into X$ and $\frak j:X[\tfrac1p]\to X$ respectively.

By taking a filtered colimit we may reduce to the same assertion in which $R$ is replaced by a local, essentially smooth $V$-algebra $S$, at which point the claim follows by taking global sections on $\Spec(S/\frak mS)$ of Bloch--Kato's residue map of \'etale sheaves $\bor:\frak i^*R^j\frak j_*\mu_{p^r}^{\otimes j}\to W_r\Omega_{\Spec(S/\frak mS),\sub{log}}^{j-1}$ \cite[(6.6)]{Bloch1986}. We note that Bloch--Kato's construction of this map can also be replaced by an argument through Gersten complexes \cite[Lem.~3.2.4]{Sato2007}. 
\end{proof}

\begin{remark}[Symbolic generation of $W_r\Omega^j_\sub{log}$]\label{remark_WrOmegalog}
Let $A$ be a local $\bb F_p$-algebra. We denote by $W_r\Omega^j_{A,\sub{log}}$ the subgroup of $W_r\Omega^j_A$ generated by $\dlog{[f_1]}\wedge\cdots\wedge\dlog{[f_j]}$ for $f_1,\dots,f_j\in A^\times$, where $[f]\in W_r(A)$ is the Teichm\"uller lift of any $f\in A^\times$ and $\dlog{[f]}:=\tfrac{d[f]}{[f]}$. This coincides with the more common definition, often denoted by $\nu_r(j)(A)$, in terms of forms which are \'etale locally generated by such dlog forms by \cite[Corol.~4.2(i)]{Morrow_pro_GL2}. 

The symbol map $K_j^M(A)\to W_r\Omega^j_{A,\sub{log}}$ descends to improved Milnor $K$-theory. Indeed, the map $W_r\Omega^j_{A,\sub{log}}\to \prod_{\frak p\in\Spec A}W_r\Omega^j_{A_{\frak p}^\sub{sh},\sub{log}}$ is injective, where $A_{\frak p}^\sub{sh}$ is the strict henselisation of $A_\frak p$, so it is enough to show that the symbol map descends for each $A_\frak p^\sub{sh}$; but then there is nothing to prove as $K_j^M(A_\frak p^\sub{sh})\isoto \hat K_j^M(A_\frak p^\sub{sh})$.

If $A$ is moreover assumed to be regular Noetherian (or, more generally, ind smooth over $\bb F_p$) then the symbol map $\dlog:\hat K_j^M(A)/p^r\to W_r\Omega^j_{A,\sub{log}}$ is an isomorphism by the Bloch--Kato--Gabber theorem \cite{Bloch1986} and the Gersten conjectures for both sides \cite{GrosSuwa1988, Kerz2009} (see \cite[Thm.~5.1]{Morrow_pro_GL2} for a more detailed discussion of the proof).
\end{remark}

The second formulation of Theorem \ref{theorem_BK} eliminates the hypothesis that $R$ has big residue field:

\begin{theorem}\label{theorem1}
Let $V$ be a complete discrete valuation ring of mixed characteristic, and $R$ a local, $p$-henselian, ind-smooth $V$-algebra; let $j,r\ge 1$. Then the sequence \[0\To \hat K_j^M(R)/p^r\To H^j_\sub{\'et}(R[\tfrac1p],\mu_{p^r}^{\otimes j})\xto{\bor} W_r\Omega_{R/\frak m R,\sub{log}}^{j-1}\To 0\] is exact.
\end{theorem}

We will see in the next section that Theorem \ref{theorem_BK} and Theorem \ref{theorem1} are indeed equivalent if $R$ has big residue field.

\begin{remark}
The assumption that $V$ is complete in Theorems \ref{theorem_BK} and \ref{theorem1} is actually redundant. Indeed, all terms in the conclusions of the theorems are unaltered if we replace $R$ by the $p$-henselisation of $R\otimes_V\hat V$ (c.f., the first paragraph of the proof of Theorem \ref{theorem_Nesterenko_Suslin}).
\end{remark}

\subsection{Reduction to the mod \texorpdfstring{$p$}{p}, absolutely unramified, big residue field case}
In this section we reduce Theorems \ref{theorem_BK} and \ref{theorem1} to the special case of Theorem \ref{theorem_BK} where $V$ has big residue field, is absolutely unramified, and $r=1$ (see Corollary \ref{corollary_final_reduction}) which will then be proved in Section \ref{section_unramified_p} (see Theorem \ref{theorem_BK_unram}).

We begin by reducing Theorem \ref{theorem1} to the case of big residue field:

\begin{proposition}\label{prop_reduction1}
Theorem \ref{theorem1} reduces to the case that $V$ has big residue field (i.e., its residue field has more than $M_j$ elements).
\end{proposition}
\begin{proof}
Obviously we only need to worry about the case that $V$ has finite residue field $k$. By taking a filtered colimit we reduce to the case that $R$ is the $p$-henselisation of a local, essentially smooth $V$-algebra; let $K$ denote its residue field, which is a finitely generated, separable field extension of $k$. So we may realise $K$ as a finite separable extension of a rational function field $k(\ul t):=k(t_1,\dots,t_d)$.

Pick any integer $\ell\ge M_j$ which is coprime to both $p$ and $|K:k(\ul t)|$, and let $k'$ be the unique degree $\ell$ extension of the finite field $k$. Note that $K\otimes_kk'=K\otimes_{k(\ul t)} k'(\ul t)$ is the tensor product of finite field extensions of coprime degree, hence is a field.

Let $V'$ be the finite \'etale extension of $V$ corresponding to the extension $k'$ of $k$, and set $R':=R\otimes_VV'$. Note that $R'$ is still $p$-henselian, since it is a finite extension of $R$, and that it is local since $R'/\frak m_RR'=K\otimes_kk'$ is a field. In particular, we are allowed to assume that Theorem \ref{theorem1} holds for $R'$.

But now we obtain Theorem \ref{theorem1} for $R$ by an easy norm argument as follows. There is a map of complexes
\[\xymatrix{
0\ar[r] & \hat K_j^M(R')/p^r\ar[r]\ar@{-->}@/^5mm/[d]^N & H^j_\sub{\'et}(R'[\tfrac1p],\mu_{p^r}^{\otimes j})\ar[r]\ar@{-->}@/^5mm/[d]^N& W_r\Omega_{R'/\pi R',\sub{log}}^{j-1}\ar[r] &0\\
0\ar[r] & \hat K_j^M(R)/p^r\ar[u]\ar[r]& H^j_\sub{\'et}(R[\tfrac1p],\mu_{p^r}^{\otimes j})\ar[r]\ar[u]& W_r\Omega_{R/\pi R,\sub{log}}^{j-1}\ar[u]\ar[r] &0
}\]
in which we include the norm maps as dashed arrows. The left vertical arrow is injective since its composition with $N$ is multiplication by $\ell$, so we deduce that $\hat K_j^M(R)/p^r\to H^j_\sub{\'et}(R[\tfrac1p],\mu_{p^r}^{\otimes j})$ is injective. Then exactness at the middle of the bottom row follows from the analogous exactness of the top row and compatibility of the norm maps (see Lemma \ref{lemma_norms_are_comp}). Since Kato's residue map is surjective (see the proof of Proposition \ref{prop_reduction2}) we have proved exactness of the bottom row, as desired.
\end{proof}

As promised in Remark \ref{remark_improved} and used in the previous proof, here is the required compatibility of the norm maps on improved Milnor $K$-theory and \'etale cohomology:

\begin{lemma}\label{lemma_norms_are_comp}
Let $R$ be a local, $p$-henselian, ind-smooth algebra over a complete discrete valuation ring of mixed characteristic; let $S$ be a local, finite \'etale $R$-algebra. Then the diagram
\[\xymatrix{
\hat K_j^M(S)\ar[r]\ar[d]_N & H^j_\sub{\'et}(S[\tfrac1p],\mu_{p^r}^{\otimes j})\ar[d]^N\\
\hat K_j^M(R)\ar[r] & H^j_\sub{\'et}(R[\tfrac1p],\mu_{p^r}^{\otimes j})
}\]
commutes for any $j,r\ge0$.
\end{lemma}
\begin{proof}
To reduce the compatibility to the case of fields we let $F:=\Frac(R_{\frak mR}^h)$ be as in the paragraph before Lemma \ref{lemma_Katos_residue} and set $L:=S\otimes_RF$, which we claim is a field. Since $L$ is finite \'etale over the field $F$, it is enough to show that it is an integral domain, which will follow from showing that $S\otimes_RR_{\frak mR}^h$ is an integral domain. But this is finite \'etale over the henselian discrete valuation ring $R_{\frak mR}^h$, so it is an integral domain if and only if its special fibre $S\otimes_RR_{\frak mR}/\frak mR_{\frak mR}$ is local. But this is finite \'etale over the field $R_{\frak mR}/\frak mR_{\frak mR}$ so (conversely to above) it is enough to check it is an integral domain; but it is a localisation of $S/\frak mS$, which is an integral domain since it is local and ind-smooth over the field $V/\frak m$.

There is a cube of maps
\[\xymatrix@=.5cm{
&&&\hat K_j^M(L)\ar[r]\ar[dd]_N & H^j_\sub{\'et}(L,\mu_{p^r}^{\otimes j})\ar[dd]^N\\
\hat K_j^M(S)\ar[rrru]\ar[r]\ar[dd]_N & H^j_\sub{\'et}(S[\tfrac1p],\mu_{p^r}^{\otimes j})\ar[rrru]\ar[dd]^N&&&\\
&&&\hat K_j^M(F)\ar[r] & H^j_\sub{\'et}(F,\mu_{p^r}^{\otimes j})\\
\hat K_j^M(R)\ar[r]\ar[rrru] & H^j_\sub{\'et}(R[\tfrac1p],\mu_{p^r}^{\otimes j})\ar[rrru] &&&
}\]
in which the goal is to prove that the front face commutes. The back left face commutes by the compatibility of Kerz' norm maps for $R\to S$ and $F\to L$; the back right face commutes by compabibility of the norm maps on Milnor $K$-theory and \'etale cohomology in the case of fields \cite[Prop. 7.5.5]{Gille2006}. Since the top, bottom, and front right faces clearly commute, the desired commutativity of the front right face now follows formally from the fact that the map $H^j_\sub{\'et}(R[\tfrac1p],\mu_{p^r}^{\otimes j})\to H^j_\sub{\'et}(F,\mu_{p^r}^{\otimes j})$ is injective: this is a result of Gabber which we recall as Theorem \ref{theorem_Gabber_injectivity}.
\end{proof}

To compare the two main theorems, we will use the following localisation sequence; we record it in greater generality than immediately necessary:

\begin{lemma}\label{lemma_localisation_in MilnorK}
Let $R$ be a ring additively generated by units and $t\in R$ a non-zero-divisor in the Jacobson radical such that $tR$ is a prime ideal, $R_{tR}$ is a discrete valuation ring,\footnote{We record the following criterion: given a ring $R$ and non-zero-divisor $t\in R$, then $R_{t R}/\bigcap_{r\ge 1}t^rR_{t R}$ is a discrete valuation ring (so, in particular, if $R_{t R}$ is $t$-adically separated then it is a discrete valuation ring). Indeed, $R_{t R}$ is a local ring with maximal ideal generated by non-zero-divisor $\pi$, and then it is easily to check $\pi$ is still a non-zero-divisor in the quotient $R':=R_{t R}/\cap_{r\ge 1}t^rR_{t R}$. So the latter is a local ring, with maximal ideal generated by a single non-zero-divisor $t$, such that $\max\{r\ge0:x\in t^rR'\}$ exists for all non-zero $x\in R'$; it easily follows that $R'$ is a dvr.} and $R[\tfrac1t]\cap R_{tR}=R$. Then there are unique homomorphisms $\bor$ and $\bor_t$ fitting into a commutative diagram with exact row
\[\xymatrix{
K_j^M(R)\ar@{->>}[d]\ar[r] & K_j^M(R[\tfrac1t])\ar[dl]^{\bor_t}\ar[r]^\bor & K_{j-1}^M(R/tR)\ar[r]&0\\
K_j^M(R/tR)&&&
}\]
and satisfying \[\bor(\{a_1,\dots,a_{j-1},t\})=\{a_1,\dots,a_{j-1}\},\qquad \bor_t(\{a_1,\dots,a_{j-1},t\})=0\] for all $a_1,\dots,a_{j-1}\in R^\times$.
\end{lemma}
\begin{proof}
The hypotheses ensure that any unit of $R[\tfrac1t]$ may be written uniquely as $t^iu$ for some $i\in\bb Z$ and $u\in R^\times$. The maps $\bor$ and $\bor_t$ may then be constructed via the usual argument due to Serre which is well-known in the case of a discrete valuation ring with uniformiser $t$; e.g., see \cite[Prop.~7.1.4]{Gille2006}. 

To see that the row is exact one considers the map \[K_{j-1}^M(R/tR)\To K_j^M(R[\tfrac1t])/\op{Im}(K_j^M(R)\to K_j^M(R[\tfrac1t])),\qquad \{a_1,\dots,a_{j-1}\}\mapsto \{\tilde a_1,\dots,\tilde a_{j-1},t\},\] where the $\tilde a_i\in R^\times$ are arbitrary lifts of $a_i\in (R/tR)^\times$. It is enough to show that this is well-defined, as it will then provide the desired inverse to $\bor$ modulo the image of $K_j^M(R)$; the well-definedness reduces to checking that $\{1+bt,t\}\in\op{Im}(K_2^M(R)\to K_2^M(R[\tfrac1t])$ for all $b\in R$. But $R$ is additively generated by units, so the group $1+tR$ is generated multiplicatively by its subset $1+tR^\times$ (Proof: by induction on the length of the expression for $b$ as a sum of units, noting that if $b=b'+u$ with $u$ a unit then $1+bt=(1+b't)(1+\tfrac{u}{1+b't}t)$.) and therefore we may assume $b$ is a unit; then $\{1+bt,t\}=\{1+bt,-b\}$ indeed lies in the image.
\end{proof}

\begin{proposition}\label{prop_reduction2}
For any fixed $V$, $R$, $j$, $r$, where $R$ has big residue field, Theorems \ref{theorem_BK} and \ref{theorem1} are equivalent.
\end{proposition}
\begin{proof}
We have a commutative diagram of complexes
\[\xymatrix{
0\ar[r] &\hat K_j^M(R)/p^r\ar[r] & H^j_\sub{\'et}(R[\tfrac1p],\mu_{p^r}^{\otimes j})\ar[r]& W_r\Omega_{R/\pi R,\sub{log}}^{j-1}\ar[r] & 0\\
0\ar[r] &K_j^M(R)/p^r\ar[r]\ar[u]^\cong & K_j^M(R[\tfrac1p])/p^r\ar[u]^{h_{p^r}^j}\ar[r]_-{\bor}\ar[u]& K_{j-1}^M(R/\pi R)/p^r\ar[r]\ar[u]^\cong & 0
}\]
The two indicated vertical arrows are isomorphisms since $R$ and $R/\pi R$ have big residue field (for the isomorhism on the right, note that $R/\pi R$ is an ind-smooth, local $\bb F_p$-algebra and see Remark \ref{remark_WrOmegalog}). From a diagram chase we now see that the top complex is exact if and only if (1) the bottom complex is also injective at the left and (2) the Galois symbol is an isomorphism.
\end{proof}

The following type of reduction to the mod $p$ case is well-known, though we are forced to modify the usual argument since the Milnor $K$-theory of the non-local ring $R[\tfrac1p]$ does not a priori admit norm maps:

\begin{proposition}\label{prop_reduction0}
Let $V$ be a complete discrete valuation ring of mixed characteristic; letting $F$ denote its field of fractions, set $F':=F(\zeta_p)$ and let $V'$ be the ring of integers of $F'$. Assume that $f(F'/F)=1$, i.e., that no residue field extension occurs. Let $R$ be a local, $p$-henselian, ind-smooth $V$-algebra whose residue field has $\ge M_J$ elements for some fixed $J\ge1$, and put $R':=R\otimes_VV'$. Then 
\begin{center}
\begin{tabular}{c}
``Theorem \ref{theorem_BK} for $V\to R$, for $r=1$, and for $0\le j\le J$\\
and\\
Theorem \ref{theorem_BK} for $V'\to R'$, for $r=1$, and for $0\le j\le J$''\\
implies\\
Theorem \ref{theorem_BK} for $V\to R$, for all $r\ge1$, and for $0\le j\le J$.
\end{tabular}
\end{center}
\end{proposition}
\begin{proof}
Let $R$ be a local, $p$-henselian, ind-smooth $V$-algebra, and set $R':=R\otimes_VV'$. Note that $R'$ is ind-smooth over $V'$, that it is $p$-henselian (since it is a finite $R$-algebra), and consequently that it is local (since $R'/\frak m'R'=R/\frak mR\otimes_{V/\frak m}V'/\frak m'=R/\frak mR$ is local, where we crucially use our hypothesis that $V$ and $V'$ have the same residue field). 

Assuming Theorem \ref{theorem_BK} for $r=1$, $j\ge0$, and both $V\to R$ and $V'\to R'$, we will actually prove Theorem \ref{theorem_BK} for $r\ge1$, for $0\le j\le J$, and for both $V\to R$ and $V'\to R'$; this looks stronger than stated in the proposition but is actually equivalent, since the statement of the proposition can also be applied to the case $V=V'$). We do this by induction on $r$.

We begin by proving part (i) of Theorem \ref{theorem_BK}. We will use the commutative diagram in which the rows are exact:
\[\xymatrix{
\cdots\ar[r]^-\delta & H^j_\sub{\'et}(R[\tfrac1p],\mu_{p^{r-1}}^{\otimes j})\ar[r]^\rho& H^j_\sub{\'et}(R[\tfrac1p],\mu_{p^r}^{\otimes j})\ar[r] & H^j_\sub{\'et}(R[\tfrac1p],\mu_{p}^{\otimes j})\ar[r] & 0 \\
& K_j^M(R[\tfrac1p])/p^{r-1}\ar[u]^\cong \ar[r]_p& K_j^M(R[\tfrac1p])/p^{r}\ar[u]\ar[r] & K_j^M(R[\tfrac1p])/p \ar[r]\ar[u]^\cong &0 
}\]
The indicated vertical arrows are isomorphisms by the inductive hypothesis, the right one of which is used to obtain right exactness of the top sequence. It follows that the middle vertical arrow is surjective. Writing $\delta'$ for the boundary map in the analogous diagram for $R'$, there is a commutative diagram
\begin{equation}\xymatrix{
H^{j-1}_\sub{\'et}(R'[\tfrac1p],\mu_{p}^{\otimes j})\ar[r]^{\delta'} & H^j_\sub{\'et}(R'[\tfrac1p],\mu_{p^{r-1}}^{\otimes j})\\
\mu_p\otimes_{\bb F_p}K_{j-1}^M(R'[\tfrac1p])/p\ar[u]^{[\,\,]\cup h^{j-1}_{p^{r-1}}}_\cong\ar[r]_-\gamma & K_j^M(R'[\tfrac1p])/p^{r-1}\ar[u]^{h^j_{p^{r-1}}}_\cong
}\label{eqn_delta}\end{equation}
where $\gamma$ denotes multiplication and $[\,\,]:\mu_p(V')\to H^0_\sub{\'et}(R'[\tfrac1p],\mu_{p^{}}^{})$ is the canonical map, where we write $\mu_p(V')$ for the group of $p^\sub{th}$ roots of unity in $V'$. The commutativity of this diagram is proved exactly as in the case of a field, for which we refer to \cite[Lem. 7.5.10]{Gille2006}.

The vertical arrows in (\ref{eqn_delta}) are isomorphisms by the inductive hypothesis and the fact that the top left term can equivalently be written $\mu_p( V')\otimes_{\bb F_p}H^{j-1}_\sub{\'et}(R'[\tfrac1p],\mu_{p}^{\otimes j-1})$. A diagram chase now proves the inductive step in the case of $V'\to R'$ 

Now let $\Delta$ be the Galois group of the extension $\Frac V\subseteq \Frac V'$, and replace the commutative square (\ref{eqn_delta}) by its $\Delta$ invariants; the existence of trace maps on \'etale cohomology and the coprimeness of $p,|\Delta|$ imply that $H^j_\sub{\'et}(R'[\tfrac1p],\mu_{p^{r-1}}^{\otimes j})^\Delta=H^j_\sub{\'et}(R[\tfrac1p],\mu_{p^{r-1}}^{\otimes j})$ (and similarly for the top left term), whence the resulting commutative square is 
\begin{equation}\xymatrix{
H^{j-1}_\sub{\'et}(R[\tfrac1p],\mu_{p}^{\otimes j})\ar[r]^{\delta} & H^j_\sub{\'et}(R[\tfrac1p],\mu_{p^{r-1}}^{\otimes j})\\
(\mu_p\otimes_{\bb F_p}K_{j-1}^M(R'[\tfrac1p])/p)^\Delta\ar[u]_\cong\ar[r]_-\gamma & K_j^M(R[\tfrac1p])/p^{r-1}\ar[u]_\cong
}\label{eqn_delta2}\end{equation}
We emphasise that we are not using the a priori existence of any trace maps on Milnor $K$-theory to obtain the identification $(K_j^M(R'[\tfrac1p])/p^{r-1})^\Delta=K_j^M(R[\tfrac1p])/p^{r-1}$, but rather the inductive hypothesis to identify both terms with \'etale cohomology (although these identifications in fact define the trace maps which could alternatively be used). The surjectivity of the left arrow, combined with our initial diagram and inductive hypothesis, shows at once that the middle symbol map in the initial diagram is injective, as required to complete the proof of the reduction.

Part (ii): We claim first that the map $K_j^M(R)\to K_j^M(R')$ is injective modulo any power of $p$. Indeed, by a filtered colimit argument we may assume that $R$ is in addition regular Noetherian, in which case it satisfies the conditions of Dahlhausen's version of Kerz' norm map \cite[Prop.~C.1]{Dahlhausen2018}: indeed, $R$ is factorial by Auslander--Buchsbaum, and $V'=V[X]/\pi$ for some monic irreducible polynomial $\pi\in V[X]$ which remains irreducible in $R[X]$ (since otherwise $R'=R[X]/\pi$ would not be a domain). So there is a norm map $\hat K_j^M(R')\xto{N}\hat K_j^M(R)$ such that the composition $K_j^M(R)=\hat K_j^M(R)\to \hat K_j^M(R')\xto{N}\hat K_j^M(R)=K_j^M(R)$ is multiplication by an integer coprime to $p$. This proves the claimed injectivity, thereby reducing the desired injectivity of $K_j^M(R)/p\to K_j^M(R[\tfrac1p])/p^r$ to that of $K_j^M(R')/p\to K_j^M(R'[\tfrac1p])/p^r$, which we will check by induction on $r$.

To do this we will use the commutative diagram
\[\xymatrix{
\mu_p\otimes_{\bb F_p}K_{j-1}^M(R'[\tfrac1p])/p\ar[r]_-\gamma & K_j^M(R'[\tfrac1p])/p^{r-1} \ar[r]_p& K_j^M(R'[\tfrac1p])/p^{r}\ar[r] & K_j^M(R'[\tfrac1p])/p \ar[r] &0 \\
& K_j^M(R')/p^{r-1}\ar@{^(->}[u] \ar[r]_p& K_j^M(R')/p^{r}\ar[u]\ar[r] & K_j^M(R')/p \ar[r]\ar@{^(->}[u]&0
}\]
in which the top row is exact by the proof of part (i). The indicated arrows are injective by the inductive hypotheses. Suppose that $\al\in K_j^M(R')/p^r$ is an element which vanishes in $K_j^M(R'[\tfrac1p])/p^r$. A diagram chase immediately shows the following (bearing in mind that any element of $K_{j-1}^M(R'[\tfrac1p])$ may be written as $\beta+\{\pi\}\beta'$ for some $\beta\in K_{j-1}^M(R')$, $\beta'\in K_{j-2}^M(R')$): the element $\al$ may be written as $p\al'$, where $\al'\in K_j^M(R')/p^{r-1}$ has image $\{\zeta_p\}\beta+\{\zeta_p,\pi\}\beta'$ in $K_j^M(R'[\tfrac1p]/p^{r-1})$. But $\zeta_p=1+\pi a$ for some $a\in  V'$, so the second term may be rewritten $\pid{a,\pi}\beta'$ and the injectivity of the left vertical arrow shows that $\al'=\{\zeta_p\}\beta+\pid{a,\pi}\beta'$. Multiplication by $p$, i.e., the lower left horizontal map, now maps $\alpha'$ to $\al=p(\{\zeta_p\}\beta+\pid{a,\pi}\beta')=p\pid{a,\pi}\beta'$. But this term is zero since $p\pid{a,\pi}$ is zero in $K_2^M(R')/p^r$: indeed certainly $p\pid{a,\pi}=p\{\zeta_p,\pi\}$ vanishes in $K_2^M(R'[\tfrac1p])$, and $K_2^M(R')/p^r\to K_2^M(R'[\tfrac1p])/p^r$ is injective by Remark \ref{remark_j=2}.
\end{proof}

It remains to reduce further to the case that $V$ is absolutely unramified. This is achieved via a left Kan extension argument using motivic cohomology. For any $p$-adically complete ring $A$, or more generally derived $p$-complete simplicial ring, we denote by $\bb Z_p(j)(A)$ its weight-$j$ \'etale-syntomic cohomology in the sense of \cite{BhattMorrowScholze2}; more precisely, this was defined in \cite{BhattMorrowScholze2} for quasisyntomic rings, then extended via left Kan extension to derived $p$-complete simplicial rings in \cite[\S5]{AntieauMathewMorrowNikolaus}; we refer to op.~cit.\ for further details; in particular, each $\bb Z_p(j)(A)$ is a $p$-complete complex supported in cohomological degree $\le j+1$. To avoid constantly needing to include additional completions, it will be convenient to write \[\bb Z_p(j)(A):=\bb Z_p(j)(\op{Rlim}_sA\dotimes_{\bb Z}\bb Z/p^s\bb Z)\]
for any $p$-henselian ring, where $\op{Rlim}_sA\dotimes_{\bb Z}\bb Z/p^s\bb Z$ is the derived $p$-adic completion of $A$; if $A$ has bounded $p$-power torsion, in particular if $A$ is Noetherian, this coincides with the usual $p$-adic completion $\hat A$.

For smooth algebras over complete discrete valuation rings, forthcoming work of Bhatt--Clausen--Mathew concerning the motivic filtration on Selmer $K$-theory will include the following identification of $\bb Z_p(j)$ with the original approach to $p$-adic \'etale motivic cohomology studied by Geisser, Sato, and Schneider \cite{Geisser2004, Sato2007, Schneider1994}:

\begin{theorem}[Bhatt--Clausen--Mathew]\label{theorem_BCM_LKE}
Let $V$ be a mixed characteristic complete discrete valuation ring and $R$ a $p$-henselian, ind-smooth $V$-algebra. Then there are natural equivalences \[\bb Z_p(j)(R)/p^r\simeq \op{hofib}\big(R\Gamma_\sub{\'et}(R/\frak mR,\frak i^*\tau^{\le j}R\frak j_*\mu_{p^r}^{\otimes j})\To R\Gamma_\sub{\'et}(R/\frak mR,W_r\Omega^{j-1}_{\sub{log}})[-j]\big)\] for all $r,j\ge1$, where the arrow is Bloch--Kato's residue map as in the proof of Lemma \ref{lemma_Katos_residue}.
\end{theorem}

\begin{corollary}\label{corollary_BCM}
Let $V$ be a mixed characteristic complete discrete valuation ring and $R$ a $p$-henselian, ind-smooth $V$-algebra. Then there are natural isomorphisms \[H^j(\bb Z_p(j)(R)/p^r)\cong \op{ker}(H^j_\sub{\'et}(R[\tfrac1p],\mu_{p^r}^{\otimes j})\xto{\bor} W_r\Omega_{R/\frak m R,\sub{log}}^{j-1})\] for all $r,j\ge1$.
\end{corollary}
\begin{proof}
This follows from Theorem \ref{theorem_BCM_LKE} and Gabber's affine analogue of the proper base change theorem implying that $R\Gamma_\sub{\'et}(R/\frak mR,\frak i^*\tau^{\le j}R\frak j_*\mu_{p^r}^{\otimes j})\simeq R\Gamma_\sub{\'et}(R,\tau^{\le j}R\frak j_*\mu_{p^r}^{\otimes j})$.
\end{proof}

The second property of the $\bb Z_p(j)(-)$ which we need to recall is that they are suitably left Kan extended:

\begin{proposition}\label{proposition_motivic_cohomol_LKE}
For each $j,r\ge 1$, the functor $\{\text{$p$-henselian rings}\}\to D(\bb Z/p^r\bb Z)$, $A\mapsto \tau^{\le j}(\bb Z_p(j)(A)/p^r)$ is left Kan extended from $p$-henselian, ind-smooth $\bb Z_{(p)}$-algebras.
\end{proposition}
\begin{proof}
Dropping the truncation initially, the functor $\bb Z_p(j)(-)/p^r$ itself is left Kan extended from $p$-henselisations of finitely generated $\bb Z_{(p)}$-polynomial algebras (hence a fortiori from $p$-henselian, ind-smooth $\bb Z_{(p)}$-algebras) thanks to the analogous result in the $p$-complete case \cite[Thm.~5.1(2)]{AntieauMathewMorrowNikolaus}.

It therefore remains to check that $\tau^{>j}(\bb Z_p(j)(-)/p^r)$ is left Kan extended from $p$-henselian, ind-smooth $\bb Z_{(p)}$-algebras. For any fixed $p$-henselian $A$, it is equivalent by cofinality to show that the canonical map $\op{hocolim}_{R\to A}\tau^{>j}(\bb Z_p(j)(R)/p^r)\to \tau^{>j}(\bb Z_p(j)(A)/p^r)$ is an equivalence, where the colimit is taken over $p$-henselian, ind-smooth $\bb Z_{(p)}$-algebras $R$ equipped with morphism $R\to A$ which is a henselian surjection. But this follows by noting the much stronger fact that each map $\tau^{>j}(\bb Z_p(j)(R)/p^r)\to\tau^{>j}(\bb Z_p(j)(A)/p^r)$ is an equivalence: indeed, by taking a filtered colimit it is enough to check that $\tau^{>j}(\bb Z_p(j)(R)/p^r)\quis\tau^{>j}(\bb Z_p(j)(A)/p^r)$ whenever $R\to A$ is a henselian surjection of $p$-henselian rings with bounded $p$-power torsion; but then $\hat R\to\hat A$ is also a henselian surjection (we leave this simple verification to the reader) and indeed $\tau^{>j}(\bb Z_p(j)(\hat R)/p^r)\quis \tau^{>j}(\bb Z_p(j)(\hat A)/p^r)$ by the rigidity property of $\bb Z_p(j)$ \cite[Thm.~5.2]{AntieauMathewMorrowNikolaus}.
\end{proof}

We will combine this with the following left Kan extension property of Milnor $K$-theory; we refer the reader to \cite[App.~A]{ElmantoHoyoisKhanSosniloYakerson2019} for further information about left Kan extending from smooth algebras.

\begin{proposition}\label{proposition_LKE_Milnor}
Let $k$ be a commutative ring and $\ell\ge 0$ an integer; let $L^\sub{sm}(K_j^M/\ell):\op{Alg}_k^\sub{loc}\to D^{\le 0}(\bb Z)$ be the left Kan extension of $K_n^M(-)/\ell$ from local, ind-smooth $k$-algebras. Then, for any local $k$-algebra $A$, the co-unit map $H_0(L^\sub{sm}(K_j^M/\ell)(A))\to K_j^M(A)/\ell$ is an isomorphism; i.e., Milnor $K$-theory is ``right exact''.
\end{proposition}
\begin{proof}
We use a standard type of argument going back to Quillen, showing that functors with suitable universal properties are right exact. To simplify the notation we write $k_j^M=K_j^M/\ell$. We will exploit the fact that the left Kan extension is lax symmetric monoidal, so that $\bigoplus_{j\ge0}H_0(L^\sub{sm}k_j^M(A))$ naturally has the structure of a graded commutative ring.

We first handle the case $j=1$. In that case more is known: the functor $K_1^M=\bb G_m$ is left Kan extended from local, ind-smooth $K$-algebras by \cite[Prop. A.0.1]{ElmantoHoyoisKhanSosniloYakerson2019}, so taking $H_0$ is not even necessary if $\ell=0$. When $\ell>0$ the right exact sequence of functors $K_1^M\xto{\ell}K_1^M\to k_1^M\to 0$ on $\op{Alg}_k^\sub{loc}$ induces a right exact sequence of abelian groups $L^\sub{sm}K_1^M(A)\xto{\ell}L^\sub{sm}K_1^M(A)\to L^\sub{sm}k_1^M(A)\to 0$; from the previous sentence we deduce that $H_0(L^\sub{sm}k_1^M(A))=A^\times/\ell$, as required.

Let $R_\bullet\to A$ be a simplicial resolution by ind-smooth, local $k$-algebras where, for each simplicial degree $q\ge0$, the kernel of the surjection $R_q\to A$ is a henselian ideal. Such a resolution exists and calculates our desired left Kan extension; namely, given any functor $F:\op{Sm}_k^\sub{loc}\to D(\bb Z)$ commuting with filtered colimits, the value of its left Kan extension $L^\sub{sm}F$ on the arbitrary local $k$-algebra $A$ is given by the geometric realisation of the simplicial object $q\mapsto F(R_q)$. In particular, $L^\sub{sm}k_j^M(A)\in D(\bb Z)$ can be represented by (the complex associated to) the simplicial abelian group $k_j^M(R_\blob):q\mapsto k_j^M(R_q)$.

We will argue using the following commutative diagram of graded commutative rings
\[\xymatrix{
\bigoplus_{j\ge0}k_j^M(R_0) \ar[d] \ar[rd] & \\
\bigoplus_{j\ge 0}H_0(L^\sub{sm}k_j^M(R_\blob)) \ar[r] & \bigoplus_{j\ge0}k_j^M(A)
}\]
in which all arrows are surjective (the vertical as $H_0(k_j^M(R_\blob))$ is by definition the equaliser of $k_j^M(R_1)\rightrightarrows k_j^M(R_0)$; the diagonal as $R_0\to A$ is surjective on units). The goal is to show that the horizontal arrow is an isomorphism.

The horizontal arrow is an isomorphism in degrees $0$ (as $k_0^M$ is the constant functor $\bb Z/\ell\bb Z$) and $1$ (explained above). Therefore $\bigoplus_{j\ge 0}H_0(L^\sub{sm}k_j^M(R_\blob))$ is the quotient of $\bigoplus_{j\ge0}k_j^M(R_0)$ by a graded ideal which is contained in $I:=\op{ker}(\bigoplus_{j\ge0}k_j^M(R_0)\to \bigoplus_{j\ge0}k_j^M(A))$ (by the existence of the commutative diagram) and contains the ideal $J$ generated by the degree one elements $\op{ker}(R_0^\times\to A^\times)$ (since the horizontal map is an isomorphism in degree $1$). But a standard argument in Milnor $K$-theory shows that $I=J$, so in fact the three ideals coincide and the bottom horizontal map is an isomorphism.
\end{proof}

We may now state and prove the final reduction:

\begin{proposition}\label{proposition_reduction_unram}
Fix $j,r\ge1$, and let $V'\subseteq V$ be an extension of mixed characteristic complete discrete valuation rings, both having big residue field. Then
\begin{center}
\begin{tabular}{c}
Theorem \ref{theorem1} for all local, $p$-henselian, ind-smooth $V'$-algebras\\ implies\\ Theorem \ref{theorem1} for all local, $p$-henselian, ind-smooth $V$-algebras.
\end{tabular}
\end{center}
\end{proposition}
\begin{proof}
Given a local, $p$-henselian, ind-smooth $V$-algebra, the description of Corollary \ref{corollary_BCM} shows that Theorem \ref{theorem1} is exactly the assertion that the symbol map defines a natural isomorphism \begin{equation}K_j^M(R)/p^r\cong H^j(\bb Z/p^r\bb Z(j)(R))=\op{ker}(H^j_\sub{\'et}(R[\tfrac1p],\mu_{p^r}^{\otimes j})\to W_r\Omega_{R/\pi R,\sub{log}}^{j-1}).\label{eqn_sub}\end{equation} Assuming that this is true for all local, $p$-henselian, ind-smooth $V'$-algebras then it follows by left Kan extending from these, using Propositions \ref {proposition_motivic_cohomol_LKE} and \ref{proposition_LKE_Milnor}, that there are natural isomorphisms \begin{equation}K_j^M(A)/p^r\isoto H^j(\bb Z/p^r\bb Z(j)(A))\label{eqn_NS}\end{equation} for all local, $p$-henselian $V'$-algebras $A$.

But when $A=R$ is a local, $p$-henselian, ind-smooth $V$-algebra then the right side is the kernel of Kato's residue map by another application of Theorem \ref{theorem_BCM_LKE}. That is, assuming that (\ref{eqn_sub}) is an isomorphism for all local, $p$-henselian, ind-smooth $V'$-algebras, we have proved it is an isomorphism for all local, $p$-henselian, ind-smooth $V$-algebras, as required.
\end{proof}

To summarise the reductions:

\begin{corollary}\label{corollary_final_reduction}
Theorems \ref{theorem_BK} and \ref{theorem1} in general reduce to the special case of Theorem \ref{theorem_BK} where $V$ has big residue field, is absolutely unramified, and $r=1$.
\end{corollary}
\begin{proof}
We suppose Theorem \ref{theorem_BK} is known whenever $V$ has big residue field, is absolutely unramified, and $r=1$. It follows from Proposition \ref{prop_reduction2} that Theorem \ref{theorem1} is true under the same conditions (obviously if $V$ has big residue field then so does $R$). For general $V$ with big residue field there exists an absolutely unramified complete discrete valuation ring $V'\subseteq V$ with the same residue field (it is the ring of Witt vector if the residue field is perfect; in the imperfect case it still exists but is not unique \cite[Chap.~2 \S5.5--5.6]{Fesenko2002}); from Proposition \ref{proposition_reduction_unram} we can then deduce that Theorem \ref{theorem1} is true whenever $V$ has big residue field and $r=1$. So Theorem \ref{theorem_BK} is true under the same conditions, again by Proposition \ref{prop_reduction2}; for any $V$ with big residue field satisfying the hypothesis of Proposition \ref{prop_reduction0}, that proposition then lets us extend Theorem \ref{theorem_BK} to all $r\ge1$. But this hypothesis holds in particular if $V$ is absolutely unramified; so appealing yet again to Proposition \ref{prop_reduction2}, we have shown that Theorem \ref{theorem1} is true whenever $V$ is absolutely unramified and has big residue field.

By again using Proposition \ref{proposition_reduction_unram} as in the previous paragraph we eliminate the hypothesis that $V$ is absolutely unramified, and then by Proposition \ref{prop_reduction1} we obtain Theorem \ref{theorem1} in general. A final application of Proposition \ref{prop_reduction2} implies Theorem \ref{theorem_BK} in general.
\end{proof}

\begin{remark}[Replacing ind-smoothness by quasismoothness]
We can actually offer the following weakening of the ind-smooth hypotheses in Theorem \ref{theorem_BK}(ii). Let $V$ be a complete discrete valuation ring of mixed characteristic, and $A$ a local, $p$-henselian $V$-algebra with big residue field satisfying the following ``quasismoothness'' conditions: $\pi A$ is a prime ideal, $A_{\pi A}$ is a discrete valuation ring, $A[\tfrac1\pi]\cap A_{\pi A}=A$, and the $\bb F_p$-algebra $A/\pi A$ is Cartier smooth (see \S\ref{subsection_KM_val_charp}). Then $K_j^M(A)/p^r\to K_j^M(A[\tfrac1p])/p^r$ is injective with cokernel $K_{j-1}^M(A/\pi A)/p^r\cong W_r\Omega^j_{A/\pi A,\sub{log}}$.

We prove this as follows. Theorem \ref{theorem_BK}(ii) and Lemma \ref{lemma_localisation_in MilnorK} provide an exact sequence \[0\To K_j^M(-)/p^r\To K_j^M(-[\tfrac1p])/p^r\To K_{j-1}^M(-/\pi -)/p^r\To 0\] of functors on the category of local, $p$-henselian, ind-smooth $V$-algebras; we may left Kan extend this to all local $p$-henselian $V$-algebras and evaluate on $A$ to obtain a fibre sequence \[L^\sub{sm}(K_j^M/p^r)(A)\To L^\sub{sm}(K_j^M(-[\tfrac1p])/p^r)(A)\To L^\sub{sm}(K_{j-1}^M(-/\pi -)/p^r)(A).\] The $H_0$ of the outer two terms are respectively $K_j^M(A)/p^r$ and $K_j^M(A/\pi A)/p^r$ by Lemma \ref{proposition_LKE_Milnor}. Meanwhile, a modification of the argument of that lemma, crucially using that $A[\tfrac1p]^\times=A^\times \pi^\bb Z$ (and similarly for any ind-smooth $V$-algebra $R$ in place of $A$), shows that $H_0$ of the middle term is $K_j^M(A[\tfrac1p])/p^r$, i.e., $K_j^M(-[\tfrac1p])/p^r$ is ``right exact at $A$''. Finally, Proposition \ref{proposition_LKE_Omegalog2}(i) shows that the rightmost term identifies with $W_r\Omega^{j-1}_{A,\sub{log}}$, and in particular it has no $H_1$. The $H_0$s of the fibre sequence therefore provide the desired short exact sequence.

Under the hypotheses on $A$, it is conceivable that the cohomological symbol $K_j^M(A[\tfrac1p])/p^r\to H^j_\sub{\'et}(A[\tfrac1p],\mu_{p^r}^{\otimes j})$ is also an isomorphism, but we have not seriously tried to prove it.
\end{remark}

\section{The unramified, mod \texorpdfstring{$p$}{p} case}\label{section_unramified_p}
In this section we prove Theorem \ref{theorem_BK} in the special case that $V$ has big residue field, is absolutely unramified, and $r=1$ (see Theorem \ref{theorem_BK_unram}). Thanks to Corollary \ref{corollary_final_reduction}, this will complete the proof of Theorems \ref{theorem_BK} and \ref{theorem1}.

\subsection{Relations in \texorpdfstring{$K_2(R)$}{K2R}}\label{subsection_K2}
Given a ring $R$, and elements $a,b\in R$ such that $1+ab\in R^\times$, let $\pid{a,b}\in K_2(R)$ denote the corresponding Dennis--Stein symbol defined in \cite{Dennis1973a}. These symbols have the following properties:
\begin{enumerate}
\item[(D1)] $\pid{a,b}=-\pid{-b,-a}$ for $a,b\in R$ such that $1+ab\in\mult R$.
\item[(D2)] $\pid{a,b}+\pid{a,c}=\pid{a,b+c+abc}$ for $a,b,c\in R$ such that $1+ab,1+ac\in\mult R$.
\item[(D3)] $\pid{a,bc}=\pid{ab,c}+\pid{ac,b}$ for $a,b,c\in R$ such that $1+abc\in\mult R.$
\end{enumerate}
We recall also the following relations to Steinberg symbols:
\begin{enumerate}
\item[(D4)] $\pid{a,b}=\{1+ab,b\}$ for $a,b\in R$ such that $b,1+ab\in R^\times$.
\item[(D5)] $\pid{a,b}=\left\{-\frac{1+a}{1-b},\frac{1+ab}{1-b}\right\}$ for $a,b\in R$ such that $1+a,1-b,1+ab\in R^\times$.
\end{enumerate}

For a deduction of (D5) from the relations (D1)--(D3), we refer to \cite[Rem.~3.14]{Morrow_K2}. Note that we use the older sign convention of \cite{Dennis1973a} for these symbols; this changed around 1980 and our $\pid{a,b}$ became $\pid{-a,b}$ (see \cite[after III, Def. 5.11]{Weibel2013}).

\begin{definition}
Let $k\geq 1$ be an integer. A ring $R$ is said to be {\em $k$-fold stable} if and only if
whenever $a_1, b_1,..., a_k, b_k\in R$ are given such that $(a_1, b_i)=\cdots =(a_k,b_k) = R$, then there exists $r \in R$ such that $a_1 + rb_1,\dots,a_k + rb_k$ are units.

Following \cite{Morrow_K2}, we will say that $R$ is {\em weakly $k$-fold stable} if and only if whenever $a_1, b_1,..., a_k, b_{k-1}\in R$ are given, then there exists $u\in R$ such that $1 + ub_1,..., 1 + ub_{k-1}$ are units. Note that weak $k$-fold stability follows from $k$-fold stability (using the pairs $(1, b_1),..., (1, b_{k-1}), (0, 1))$.
\end{definition}

\begin{remark}[Reminder on presentations of $K_2$]\label{remark_on_presentations}
\begin{enumerate}
\item If $R$ is local or three-fold stable then $K_2(R)$ is generated by the Dennis--Stein symbols subject to relations (D1)--(D3) \cite[Thm.~1]{vanderKallen1975}.
\item If $R$ is local and has residue field $\neq\bb F_2$, then $K_2(R)$ is generated by the Steinberg symbols subject to bilinearity, Steinberg relation, and the alternating relation $\{x,y\}=-\{y,x\}$ \cite[Corol.~3.2]{Kolster1985}.
\item If $R$ is weakly five-fold stable, then bilinearily and the Steinberg relation imply the skew symmetry relation $\{x,-x\}=0$ (e.g., \cite[Lem.~3.6]{Morrow_K2}, which itself implies the alternating relation; so from (ii) we deduce that if $R$ is local with residue field having $> 5$ elements, then $K_2^M(R)\isoto K_2(R)$ (even better, this isomorphism actually holds for any five-fold stable ring by combining (i) with \cite[Thm.~8.4]{vanderKallen1977}).
\item Since (D4) shows that any Steinberg symbol can be written as a Dennis--Stein symbol, we see from (i) that if $R$ is local or three-fold stable then $K_2(R)$ is generated by Steinberg symbols.
\end{enumerate}
\end{remark}

For the rest of this subsection we fix a ring $R$ which is additively generated by its units and whose Jacobson radical contains the prime number $p$.

\begin{definition}\label{definition_Vfiltration}
For $i\ge0$, let $U^iK_2(R)\subseteq K_2(R)$ denote the subgroup generated by Steinberg symbols $\{a,b\}$ where $a$ or $b$ belongs to $1+p^iR$. To give uniform statements, we also set \[V^iK_2(R):=\begin{cases} U^2K_2(R)+\langle\pid{pc,p}\,:\,c\in R\rangle & i=2=p\\ U^iK_2(R)&\text{else.}\end{cases}\] It follows from part (ii) of the next lemma that $V^2K_2(R)\subseteq U^1K_2(R)$, so this really is a descending filtration.
\end{definition}

\begin{lemma}\label{lemma_additive}
Let $i\ge1$, and fix an element $\pi\in p^iR$. Then:
\begin{enumerate}
\item $\pid{\pi, a}+\pid{\pi, b}\equiv \pid{\pi, a+ b}$ mod $U^{i}K_2(R)$ for all $a,b\in R$;
\item If $R$ is additively generated by units then $\pid{\pi,b}\in U^iK_2(R)$ for all $b\in R$.
\end{enumerate}
\end{lemma}
\begin{proof}
(i): Set $c= a b/(1+\pi( a+ b))$, so that $(1+\pi a)(1+\pi b)=(1+\pi( a+ b))(1+\pi^2c)$ and $ a+ b+\pi a b=( a+ b)+\pi c+\pi( a+ b)\pi c$. Then \[\pid{\pi, a}+\pid{\pi, b}\stackrel{\sub{(D2)}}=\pid{\pi, a+ b+\pi a b}=\pid{\pi,( a+ b)+\pi c+\pi( a+ b)\pi c}.\] Applying (D2) again lets us rewrite the right side $\pid{\pi, a+ b}+\pid{\pi,\pi c}$. The second symbol lies in $U^{i}K_2(R)$ by (D5).

(ii) By part (i) and the fact that $R$ is additively generated by units, we reduce to the case that $b$ is a unit; but then the claim is clear from (D4).
\end{proof}

\begin{lemma}\label{lemma_mult_filtration}
Let $i\ge 1$. Then
\begin{enumerate}
\item $\{1+ap^i,1+bp\}\in V^{i+1}K_2(R)$ for all $a,b\in R$;
\item $\{1+ap^i,-1\}\in V^{i+1}K_2(R)$ for all $a\in R$.
\item $\pid{p^{i-1}a,p},\pid{p^{i-1},pa}\in V^iK_2(R)$ for all $a\in R$.
\end{enumerate}
\end{lemma}
\begin{proof}
(i): All congruences below are with respect to $U^{i+1}K_2(R)$. One has \begin{equation}\{1+p^ia,1+pb\}\equiv \{1+p^ia(1+pb),1+pb\}\stackrel{\sub{(D5)}}{=}\pid{p^ia,1+pb}.\label{eqn_d5line}\end{equation} Next, a trivial induction using (D2) shows that, for any element $x\in R$ such that $1+x\in R^\times$, and $n\ge1$, we have $n\pid{x,1}=\pid{x,f(x)}$ for some polynomial $f$ with integer coefficients such that $f(0)=n$; note that this symbol equals $0$ since $\pid{x,1}=\{1+x,1\}$. In particular $0=(p-1)\pid{p^ia,1}=\pid{p^ia,p-1+p^ic}$ for some $c\in R$, which we add to the right side of (\ref{eqn_d5line}), using (D2), to obtain $\pid{p^ia,p-1+p^ic+1+pb+p^ia(p-1+p^ic)(1+pb)}=\pid{p^ia,pd}$ with $d=1+b+p^{i-1}(c+a(p-1+p^ic))$. But this equals $\pid{p^{i+1}a,d}+\pid{p^iad,p}$ by (D3), where the first term lies in $U^{i+1}K_2(R)$ by Lemma \ref{lemma_additive}(ii); meanwhile the second term is \begin{equation}\pid{p^iad,p}\stackrel{\sub{(D5)}}=\left\{-\frac{1+p^iad}{1-p},\frac{1+p^{i+1}ad}{1-p}\right\}\equiv -\{1+p^{i}ad,1-p\}.\label{eqn_penultimate}\end{equation} Assuming $i\ge 2$, then the right side is $\equiv \{1+p^{i}a(1+b),1-p\}$ thanks to the value of $d$, and so in conclusion we have shown that \[\{1+p^ia,1+pb\}\equiv -\{1+p^ia(1+b),1-p\}\] for all $a,b\in R$; but replacing the pair $a,b$ by the pair $a(1+b),0$ does not change the right side, while the left side becomes $0$, so in fact both sides are always $0$.

It remains to treat two special cases. The first is when $p$ is odd and $i=1$. Returning to line (\ref{eqn_penultimate}), it is enough to show that $\pid{pe,p}\in U^2K_2(R)$ for all $e\in R$. But \[2\pid{pe,p}\stackrel{\sub{(D3)}}=\pid{e,p^2}\stackrel{\sub{(D1)}}=-\pid{-p^2,-e}\stackrel{\sub{(D5)}}\in U^2K_2(R),\] and this is enough since $\op{gr}^1K_2(R)$ is killed by $p$ and so has no $2$-torsion.

The second special case is when $p=2$ and $i=1$. But we showed above that $\{1+p^ia,1+pb\}\equiv \pid{p^iad,p}$, which lies in $V^2K_2(M)$ by definition.

(ii): When $p$ is odd this again follows from the observation that $\op{gr}^iK_2(R)$ has no $2$-torsion. When $p=2$ it follows from part (i) by noting that $\{1+ap^i,-1\}=\{1+ap^i,1+(-1)2\}$.

(iii): When $i\ge2$ we have \[\pid{p^{i-1}a,p}\stackrel{\sub{(D5)}}=\left\{-\frac{1+p^{i-1}a}{1-p},\frac{1+p^ia}{1-p}\right\}\equiv-\{1+p^{i-1}a,1-p\} \text{ mod }U^iK_2(R),\]
which belongs to $V^iK_2^M(R)$ by part (i). Meanwhile, when $i=1$, the claim is special case of Lemma \ref{lemma_additive}(ii) (using (D1) to swap the order of the symbol).
\end{proof}

\begin{remark}\label{remark_filtration}
As far as we are aware, Lemma \ref{lemma_mult_filtration}(i) is not true in the case $p=2=i+1$ if we replace $V^2K_2(R)$ by $U^2K_2(R)$. In fact, it follows from (D5) that $V^2K_2(R)$ could equivalently be defined as $U^2K_2(R)+\langle\{1+ap,1+bp\}\,:\,a,b\in R\rangle$. Indeed, we showed in the proof of Lemma \ref{lemma_mult_filtration}(i) that $\{1+p^ia,1+pb\}\equiv \pid{p^iad,p}$ mod $U^{i+1}K_2(R)$ for some $d\in R$. The $V$-filtration is thus the minimal modification of the $U$-filtration which is multiplicative; see also Lemma \ref{lemma_multiplicative2}.
\end{remark}

\begin{corollary}\label{corollary_additive}
Let $i\ge1$, and fix an element $\pi\in p^iR$. Then $\pid{\pi, a}+\pid{\pi, b}\equiv \pid{\pi, a+ b}$ mod $V^{i+1}K_2(R)$ for all $a,b\in R$.
\end{corollary}
\begin{proof}
Repeating the proof of Lemma \ref{lemma_additive}(i), we must show that $\pid{\pi,\pi c}\in V^{i+1}K_2(R)$ for all $c\in R$. Writing $\pi=p^ib$ for some $b\in R$, we have $\pid{\pi,\pi c}=\pid{p^ib,p^ibc}\stackrel{\sub{(D3)}}=\pid{p^{2i}b,p^ibc}+\pid{p^ib^2c,p}$; the first term lies in $U^{2i}K_2(R)\subseteq V^{i+1}K_2(R)$ by Lemma \ref{lemma_additive}(ii), while the second term lies in $V^{i+1}K_2(R)$ by Lemma \ref{lemma_mult_filtration}(iii).
\end{proof}

\begin{lemma}\label{lemma_filtr_vanishes}
Assume in addition that $R$ is $p$-henselian. If $p$ is odd then $V^2K_2(R)\subseteq pK_2(R)$. If $p>2$ then $V^3K_2(R)\subseteq pK_2(R)$; if $p=2$ and the Artin--Schreier map $x\mapsto x^2+x$ is surjective on $R/pR$, then $V^2K_2(R)\subseteq pK_2(R)$.
\end{lemma}
\begin{proof}
If $p$ is odd then any element of $1+p^2R$ can be written as the $p^\sub{th}$-power of an element of $1+pR$; this is also true when $p=2$ if the Artin--Schreier map is surjective, in which case we also use the identity $-2\pid{-p,-x}=\pid{p(x^2+x),p}$ (by (D1) and (D2)) to see that $V^2K_2(R)\subseteq 2K_2(R)$.

If $p=2$ but we drop any assumption on the  Artin--Schreier map, then it is still true that any element of $1+p^3R$ can be written as the $p^\sub{th}$-power of an element of $1+p^2R$.
\end{proof}

\subsection{The differential maps to the graded pieces}\label{subsection_differential _maps}
The next step is to define the standard differential maps onto the graded pieces, for which we recall the following description of differential forms:

\begin{lemma}\label{lemma_representation_diff_forms}
Let $R$ be a ring which is additively generated by its units. Then the map \[R\otimes_{\bb Z}R^{\times\otimes_{\bb Z}j}\To \Omega_R^j,\qquad a\otimes b_1\otimes\cdots\otimes b_j\mapsto a\tfrac{db_1}{b_1}\wedge\cdots\wedge\tfrac{db_{j}}{b_{j}}\] is surjective and its kernel is generated by elements of the following types:
\begin{itemize}
\item $a\otimes b_1\otimes\cdots\otimes b_j$, where $b_i=b_j$ for some $i\neq j$, and
\item $\sum_{i=1}^n a_i\otimes a_i\otimes b_1\otimes\cdots\otimes b_{j-1}-\sum_{i=1}^m a_i'\otimes a_i'\otimes b_1\otimes\cdots\otimes b_{j-1}$ where $a_i,a_i'\in R^\times$ are elements satisfying $\sum_{i=1}^na_i=\sum_{i=1}^ma_i'$.
\end{itemize}
\end{lemma}
\begin{proof}
Used in \cite[Lem.~4.3]{Bloch1983}, a detailed proof may be found in the appendix of \cite{Izhboldin2000}.
\end{proof}

To define the differential maps in a wide degree of generality, it is helpful to adopt the following definitions (in practice our ring $R$ will be local and $\res K_*(R)$ will either be its improved Milnor $K$-theory or a quotient thereof):

\begin{definition}\label{definition_V_filtration_general}
Let $R$ be a ring which is additively generated by units and whose Jacobson radical contains $p$. Let $\res K_*(R)$ be a graded quotient of $K_*^M(R)$ such that the quotient map $K_2^M(R)\to \res K_2(R)$ kills $\ker(K_2^M(R)\to K_2(R))$.

For $i,j\ge0$, we let $U^i\res K_j(R)\subseteq \res K_j(R)$ denote the subgroup generated by Steinberg symbols $\{a_1,\dots,a_j\}$ where at least one of $a_1,\dots,a_j\in R^\times$ belongs to $1+p^iR$. Similarly to Definition \ref{definition_Vfiltration}, we also introduce
\[V^i\res K_j(R):=\begin{cases} V^2K_2(R)\cdot \res K_{j-2}(R) & p=2=i\text{ and }j\ge2, \\ U^i\res K_j(R) &\text{else.}\end{cases}\] whose graded pieces will be denoted by $\op{gr}^i_V\res K_j(R):=V^i\res K_j(R)/V^{i+1}\res K_j(R)$.
\end{definition}

Although it is not necessary for our main results, one justification for the $V$-filtration is the following multiplicativity:

\begin{lemma}\label{lemma_multiplicative2}
Under the same hypotheses as the previous definition, the $V$-filtration is descending and multiplicative, i.e., $V^i\res K_j(R)\cdot V^{i'}\res K_{j'}(R)\subseteq V^{i+i'}\res K_{j+j'}(R)$ for all $i,i',j,j'\ge0$.
\end{lemma}
\begin{proof}
We only treat the case where $p=2$, $i=2$, $i'=1$. All other cases follow from a modification of the arguments in the proof of Lemma \ref{lemma_mult_filtration}, and in any case will not be needed for the main results of the article.

Given any $a,c\in R$, we need to show that $\pid{pa,p}\{1+pc\} \in V^3K_3^M(R)$. We have that 
\begin{align*}
\pid{pa,p}\{1+pc\}&\stackrel{\sub{(D5)}}{=}\{1+pa,-(1+p^2a),1+pc\}\\ 
&=\{1+pa,1+p^2a,1+pc\}+ \{1+pa,-1,1+pc\}\\
&\stackrel{\sub{Lem.~\ref{lemma_mult_filtration}(i)}}{\equiv}  \{1+pa,-1,1+pc\}\mod{V^3\res K_3(R)}
\end{align*}
Now applying the congruence in Remark \ref{remark_filtration} to see that this is $\equiv\pid{pd,p}$ mod $U^2\res K_2(R)\{-1\}$ for some $d\in R$; but $U^2\res K_2(R)\{-1\}\subseteq V^3\res K_3(R)$ by Lemma \ref{lemma_mult_filtration}(ii), so we have reduced the  problem to showing that $\pid{pd,p}\{-1\} \in V^3K_3^M(R)$. Rewriting this using (D5) and again using Lemma \ref{lemma_mult_filtration}(ii), it is equivalent to check that $\{1+pd,-1,-1\} \in V^3K_3^M(R)$. By the same argument as in the proof of Lemma \ref{lemma_localisation_in MilnorK} we may assume that $d\in R^\times$.

Next note that we can write $1+pd=(1-pd)(1+p^2d')$ for some $d'\in R$, and therefore we have $\{1+pd,-1,-1\}=\{1+pd,-1,-d\}-\{(1-pd)(1+p^2d'),-1,d\}$. But $\{1+pd,-1,-d\}$ and $\{1-pd,-1,d\}$ vanish by Lemma \ref{lemma_K3_relation} below, and $\{1+p^2d',-1,-d\}$ lies in $V^3\res K_3^M(R)$ by Lemma \ref{lemma_mult_filtration}(ii).
\end{proof}

We next prove that Bloch--Kato's maps to the graded pieces of $p$-adic \'etale cohomology factor through the $K$-groups of $R$ itself. In the case of the map $\rho_j^i$ when $p\neq 2$, a similar result is stated by Kurihara in special cases \cite[Lem.~2.2]{Kurihara1988}, though he omits the manipulations of Dennis--Stein symbols which we believe are necessary.

\begin{definition}
Let $R$ be an $\bb F_p$-algebra. We define $\tilde\nu(j)(R):=\mathrm{coker}(1-C^{-1}:\Omega^{j}_{R}\to \Omega^{j}_{R}/d\Omega^{j-1}_{R})$, where $C^{-1}$ is the inverse Cartier operator. Equivalently $\tilde\nu(j)(R)=H^1_\sub{\'et}(\Spec R,\Omega^j_\sub{log})$, as follows from the short exact sequence $0\to\Omega^j_\sub{log}\to \Omega^j\xto{1-C^{-1}}\Omega^j/d\Omega^{j-1}\to 0$ of \'etale sheaves on $\Spec R$ \cite[Corol.~4.2(iii)]{Morrow_pro_GL2}.
\end{definition}

\begin{proposition}\label{proposition_graded_pieces_MilnorK}
Let $R$ and $\res K_*(R)$ be as in Definition \ref{definition_V_filtration_general}, and fix $j\ge0$, $i\ge1$. Then there are well-defined homomorphisms
\[\rho_j^i:\Omega^{j-1}_{R/pR}\to \op{gr}^i_V\res K_j(R),\qquad a\tfrac{db_1}{b_1}\wedge\cdots\wedge\tfrac{db_{j-1}}{b_{j-1}}\mapsto \{1+\tilde ap^i,\tilde b_1,\dots,\tilde b_{j-1}\}\] and \[\lrho_j^i:\Omega^{j-2}_{R/pR}\to \op{gr}^i_V\res K_j(R),\qquad a\tfrac{db_1}{b_1}\wedge\cdots\wedge\tfrac{db_{j-2}}{b_{j-2}}\mapsto \pid{\tilde{a}p^{i-1},p}\{\tilde b_1,\dots,\tilde b_{j-2}\}\] where tilde denotes arbitrary lifts of elements from $R/pR$ to $R$. Moreover:
\begin{enumerate}
\item When $p=2=i$ and $j\ge2$, the sum $\rho_j^i\oplus \lrho_j^i:\Omega^{j-1}_{R/pR}\oplus \Omega^{j-2}_{R/pR}\to \op{gr}^iK_j^M(R)$ is surjective; in all other cases $\rho_j^i$ itself is surjective.
\item Assume $p=2=i$, that $j\ge2$, and $\res K_*(R)$ is killed by $p$. Then $\rho_j^i$ factors through $\tilde\nu(j-1)(R/pR)$; assuming in addition that $R$ is weakly $4$-fold stable, then also $\lrho_j^i$ factors through $\tilde\nu(j-2)(R/pR)$.
\end{enumerate}
\end{proposition}
\begin{proof}
First note that $\{1+\tilde ap^i,\tilde b_1,\dots,\tilde b_{j-1}\}$ and $\pid{\tilde{a}p^{i-1},p}\{\tilde b_1,\dots,\tilde b_{j-2}\}$ do belong to $V^i\res K_j(R)$; in the first case this is by definition of the $U$-filtration, in the second case it is by Lemma \ref{lemma_mult_filtration}(iii). In other words, we at least have well-defined maps \[\tilde\rho_j^i:R\times(R^\times)^{\times j-1}\To \op{gr}^i_V\res K_j(R),\qquad a\tfrac{db_1}{b_1}\wedge\cdots\wedge\tfrac{db_{j-1}}{b_{j-1}}\mapsto \{1+ ap^i, b_1,\dots, b_{j-1}\}\] (and similarly $\tilde\lrho_j^i$), and our goal is to show that these descend to $\Omega^{j-1}_{R/pR}$ (resp.~$\Omega^{j-2}_{R/pR}$).

Firstly, multilinearity of Steinberg symbols shows that $\tilde\rho_j^i$ and $\tilde\lrho_j^i$ are multilinear away from the first coordinate. In the case of the first coordinate, $\tilde\rho_j^i$ is multilinear by definition of the $U$-filtration, while for $\tilde\lrho_j^i$ we argue as follows: for $a,a'\in R$, we first use (D1) to swap the other of the necessary relation for simplicity and then argue as in Lemma \ref{lemma_additive} to see that $\pid{p,ap^{i-1}}+\pid{p,a'p^{i-1}}=\pid{p,(a+a')p^{i-1}}+\pid{p,p^{2i-1}\tfrac{aa'}{1+p^{2i-1}(a+a')}}$, where the second term lies in $V^{i+1}\res K_2(R)$ by Lemma \ref{lemma_mult_filtration}(iii).

We next claim that the induced map on $R\otimes_{\bb Z}(R^\times)^{\otimes_{\bb Z}j-1}$ descends further to $\res R^\times\otimes_{\bb Z} (\res R^\times)^{\otimes_{\bb Z} j-1}$ (resp.~$j-2$ in place of $j-1$ in the case of $\lrho_j^i$), where we write $\res R:=R/pR$ for simplicity. For $\rho_j^i$ each of the necessary relations follows immediately either from the definition of the $U$-filtration or Lemma \ref{lemma_mult_filtration}(i), so we will focus on the slightly more complicated case of $\lrho_j^i$: firstly, if $a\in pR$ then $\pid{ap^{i-1},p}\in V^{i+1}\res K_2(R)$ by Lemma \ref{lemma_mult_filtration}(iii); secondly, if $b\in 1+pR$ then $\pid{ap^{i-1},p}\{b\}\in V^{i+1}K_3(R)$ by Lemmas \ref{lemma_mult_filtration}(iii) \& \ref{lemma_multiplicative2}.

Our desired maps $\rho_j^i$ and $\lrho_j^i$ have therefore been shown to be well-defined after pre-composition with the surjection $\res R\otimes_{\bb Z}\res R^{\times\otimes_{\bb Z}j-1}\to \Omega_{\res R}^{j-1}$, $a\otimes b_1\otimes\cdots\otimes b_j\mapsto a\tfrac{db_1}{b_1}\wedge\cdots\wedge\tfrac{db_{j-1}}{b_{j-1}}$ (resp.~$j-2$ in place of $j-1$) of Lemma \ref{lemma_representation_diff_forms}. It remains to show that the relations of that lemma are sent to zero in $\op{gr}^i_V\res K_j(R)$.

For the first type of relation, if $b_i=b_j$ in $\res R$ for some $i\neq j$ then we are free to pick the same lift $\tilde b_i=\tilde b_j$, and the desired relation follows from the fact that $\{1+\tilde ap^i,\tilde b_1,\tilde b_j\}=\{1+\tilde ap^i,\tilde b_i, -1\}\in V^{i+1}\res K_3(R)$ by Lemma \ref{lemma_mult_filtration}.

For the second type of relation, it is enough to check that the maps $R^\times \to\op{gr}_V^i\res K_2(R)$, $a\mapsto\{1+ap^i,a\}$, and $R^\times \to\op{gr}_V^i\res K_3(R)$, $a\mapsto\pid{ap^{i-1},p}\{a\}$, extend (necessarily uniquely) to additive maps defined on all of $R$. In the first case we observe that $\{1+ap^i,a\} =\pid{p^i,a}$, where the latter symbol makes sense for any $a\in R$ and is additive by Corollary \ref{corollary_additive}. In the second case, when $i=1$ the problem is vacuous as $\pid{a,p}\{a\}=-\{1+ap,-a, a\}=0$ (first equality by (D1) and (D4), second equality since $\{-a,a\}=0$ in $K_2(R)$); when $i>1$ we instead observe that
\begin{align*}
\pid{ap^{i-1},p}\{a\}&\stackrel{\sub{(D5)}}=\{-\tfrac{1+ap^{i-1}}{1-p},\tfrac{1+ap^{i}}{1-p},a\}\\
&\equiv -\{-\tfrac{1+ap^{i-1}}{1-p},1-p,a\}\mod{V^{i+1}\res K_3(R)}\\
&=\{1+ap^{i-1},a,1-p\}+\{-(1-p),1-p,a\}\\
&\stackrel{\sub{(D4)}}=\pid{p^{i-1},a}\{1-p\}
\end{align*}
where the congruence holds by Lemma \ref{lemma_mult_filtration}, and the final equality again using vanishing of $\{-(1-p),1-p\}$; the final term makes sense for any $a\in R$ and is additive modulo $V^i\res K_2(R)\{1-p\}\subseteq V^{i+1}\res K_3(R)$ by Corollary \ref{corollary_additive} and Lemma \ref{lemma_multiplicative2}.

This completes the proof that the map $\rho_j^i$ and $\lrho_j^i$ are well-defined. Claim (i) about surjectivity is immediate from the definition of the $V$-filtration, so it remains to establishing claim (ii). So in the rest of the proof we assume $p=2=i$ and $j\ge2$, and that $\res K_*(R)$ is killed by $p$.

First we show that $\rho_j^i$ vanishes on closed forms; it is enough (since $R$ is additively generated by units) to check closed forms $da\wedge \tfrac{db_1}{b_1}\wedge\cdots\wedge \tfrac{db_{j-2}}{b_{j-2}}=d(a\wedge \tfrac{db_1}{b_1}\wedge\cdots\wedge \tfrac{db_{j-2}}{b_{j-2}})$ where $a,b_1,\dots,b_{j-2}\in R^\times$. Such a form is sent to $\{1+\tilde ap^2,\tilde a,\tilde b_1,\dots,\tilde b_{j-2}\}\in\op{gr}_V^2\res K_2(R)$. But $u:=\tilde a$ is a unit and we have $\{1+up^2,u\}\equiv \{1+up^2,-u\}$ mod $V^3K_2(R)$ by Lemma \ref{lemma_mult_filtration}(ii), and then $\{1+up^2,-u\}\stackrel{\sub{(D4)}}=\pid{-p^2,-u}\stackrel{\sub{(D1)}}=-\pid{u,p^2}\stackrel{\sub{(D3)}}=-2\pid{up,p}\in 2V^2K_2(R)$, which therefore vanishes in $\op{gr}_V^2\res K_2(R)$. To show that $\rho_j^i$ descends further to $\tilde\nu(j-1)(R/pR)$, it is similarly enough to observe that $\{1+(\tilde a^2-\tilde a) p^2,\tilde b_1,\dots,\tilde b_{j-1}\}$ is a multiple of $2$ mod $V^3\res K_j^M(R)$, which is true since $1+(\tilde a^p-\tilde a)p^2= (1-\tilde ap)^2$.

Since the final sentence can also be applied to $\lrho_j^i$, it remains only to show that $\lrho_j^i$ vanishes on closed forms; similarly to the previous paragraph, it is enough to check on closed forms $-da\wedge \tfrac{db_1}{b_1}\wedge\cdots\wedge \tfrac{db_{j-3}}{b_{j-3}}$ (the minus sign is included to simplify the following manipulations; in any case $-1=1$ in $R/pR$), and this reduces to showing that $\pid{-up,p}\{u\}\in V^3\res K_3(R)$ for all $u\in R^\times$. We have
\[\pid{-up,p}\{u\}\stackrel{\sub{(D5)}}=\{1-up,-(1-up^2),u\}=\{1+up,1-up^2,-u\}+\{1+up,1-up^2,-1\}+\{1+up,-1,u\}\] (note that for the first equation we also use that $1-p=-1$) where we have already shown in the previous paragraph that $\{1+up^2,u\}\in 2V^2K_2(R)$, and in Lemma \ref{lemma_mult_filtration}(ii) that $\{1-up^2,-1\}\in V^3K_2(R)$. Therefore, to complete the proof that $\pid{-up,p}\{u\}\in V^3\res K_3(R)$ it is more than enough to show that $\{1+up,-1,u\}$ vanishes; but since $-1=1-p$, this is a special case of Lemma \ref{lemma_K3_relation} (although we assume weak $5$-fold stability in the statement of that lemma, the proof only uses weak $4$-fold stability and skew symmetry; but the latter follows in the current context from the fact that $K_2^M(R)\to \res K_2(R)$ factors through $\op{Im}(K_2^M(R)\to K_2(R))$, where skew symmetry holds).
\end{proof}

We record the following general lemma which was required:

\begin{lemma}\label{lemma_K3_relation}
Let $R$ be a weakly $5$-fold stable ring and $\pi\in R$ an element of the Jacobson radical. Then $\{1-u\pi,1-\pi,u\}=0$ in $K_3^M(R)$ for all $u\in R^\times$.
\end{lemma}
\begin{proof}
Firstly, weak $5$-fold stability implies that skew symmetry $\{x,-x\}=0$ holds for all $x\in R^\times$. So it is equivalent to prove vanishing of the symbol $\{1-u\pi,-u(1-\pi),u\}$, or even of $\{1-u\pi,u,u(\pi-1)\}$. Morally the reason that this symbol vanishes is that its terms sum to $1$ \cite[Corol.~1.8]{Suslin1991}, but we carefully unfold Suslin--Yarosh's proof to show that our stability condition suffices.

By $4$-fold stability, there exists a unit $a\in R^\times$ such that $1+a$, $1-au^{-1}$, and $1+a(1+u^{-1})$ are all units; so $u-a$ is a unit, and we set $b:=u\tfrac{a-u\pi}{a-u}\in R^\times$. Using that $\pi$ is in the Jacobson radical, it is easily checked that $a+b$ and $1+b$ are both units; note also that $\tfrac{ab}{a+b-u\pi}=u$.

It now follows from \cite[Lem.~1.5]{Suslin1991} (with $s:=u\pi$, whence $s-a$, $s-b$, and $s-a-b$ are all units) that $\{a,u\pi-a\}+\{b,u\pi-b\}=\{u,u(\pi-1)\}$. So it is enough to show that $\{1-u\pi,x,u\pi-x\}=0$ for all $x\in R^\times$ such that $1+x$ is also a unit (namely, $x=a,b$). To prove this we first recall the standard identity $\{y,x\}=\{y+x,-\tfrac{x}{y}\}$ for all $x,y\in R^\times$ for which $x+y$ is also a unit (to check this, just expand $0=\{\tfrac{y}{x+y},\tfrac{x}{x+y}\}$). Applying this with $y=1-u\pi$ we deduce that, for all $x\in R^\times$ for which $1+x$ is a unit, \[\{1-u\pi,x,u\pi-x\}=\{1-u\pi+x,-\tfrac{x}{1-u\pi},u\pi-x\},\] which vanishes as desired by the Steinberg relation.
\end{proof}

\begin{corollary}\label{corollary_graded_pieces}
Let $R$ be a $p$-henselian ring which is additively generated by units; assume that $K_2^M(R)\isoto K_2(R)$. Set $k_j^M(R):=K_j^M(R)/p$.
\begin{enumerate}
\item When $p$ is odd, there is an exact sequence \[\Omega^{j-1}_{R/pR}\xto{\rho_j^1}k_j^M(R)\xto{\bor}k_j^M(R/pR)\To 0\]
\item When $p=2$ and $R$ is weakly $4$-fold stable, then the map \[\rho^2_j\oplus\lrho^2_j:\tilde\nu(j-1)(R/pR)\oplus\tilde\nu(j-2)(R/pR)\To k_j^M(R)\] is well-defined and its cokernel fits into an exact sequence \[\Omega^{j-1}_{R/pR}\xto{\rho_j^1}\op{coker}(\rho^2_j\oplus\lrho^2_j)\xto{\bor}k_j^M(R/pR)\to 0\]
\end{enumerate}
\end{corollary}
\begin{proof}
First note that $V^1k_j^M(R)=\ker(k_j^M(R)\to k_j^M(R/pR)$. This follows from the usual construction of an inverse map $k_j^M(R/pR)\to k_j^M(R)/V^1k_j^M(R)$, $\{a_1,\dots,a_j\}\mapsto\{\tilde a_1,\dots,\tilde a_j\}$. If $p$ is odd then $V^2k_j^M(R)=0$, and if $p$ is even then $V^3k_j^M(R)=0$ (see Lemma \ref{lemma_filtr_vanishes}). The statements now follow from Proposition \ref{proposition_graded_pieces_MilnorK}.
\end{proof}

\subsection{Comparison to \'etale cohomology and proof of Theorem \ref{theorem_BK_unram}}
In this section we complete the proof of the main theorems of \ref{subsection_main_thms}. We will need the following theorem which is a special case of \cite[Corol.~1.4.1]{Bloch1986} and which calculates the sheaf of $p$-adic vanishing cycles $\frak i^*R\frak j_*\mu_p^{\otimes j}$ (for the notation see the proof of Lemma \ref{lemma_Katos_residue}) on a smooth scheme over a complete discrete valuation ring $V$ of mixed characteristic. We cite their result only in the special case when $r=1$ and $V$ is absolutely unramified, when the filtration calculating $\frak i^*R^j\frak j_*\mu_p^{\otimes j}$ is very short.

\begin{theorem}[Bloch--Kato]\label{theorem_BK_unramified}
Let $V$ be an absolutely unramified, complete discrete valuation ring of mixed characteristic, and $X$ a smooth $V$-scheme. Then there is a natural short exact sequence of \'etale sheaves on $Y$ \[0\To \Omega^{j-1}_Y\xto{\rho_j^1} \frak i^*R^j\frak j_*\mu_p^{\otimes j}\xto{\bor\oplus\bor_p}\Omega^{j-1}_{Y,\sub{log}}\oplus \Omega^j_{Y,\sub{log}}\To 0.\]
\end{theorem}

We may finally prove Theorem \ref{theorem_BK} in the unramified, mod $p$ case:

\begin{theorem}\label{theorem_BK_unram}
Let $V$ be an absolutely unramified, complete discrete valuation ring of mixed characteristic, and $R$ a local, $p$-henselian, ind-smooth $V$-algebra; let $j\ge 1$ and assume $R$ has big residue field. Then
\begin{enumerate}
\item the cohomological symbol $h^j_{p}:K_j^M(R[\tfrac1p])/p\To H^j_\sub{\'et}(R[\tfrac1p],\mu_{p}^{\otimes j})$ is an isomorphism;
\item the canonical map $K_j^M(R)/p\To K_j^M(R[\tfrac1p])/p$ is injective.
\end{enumerate}
\end{theorem}
\begin{proof}
Although $X:=\Spec R$ is not smooth over $V$, the conclusions of Theorem~\ref{theorem_BK_unramified} remain valid by taking a filtered colimit. In particular, by taking cohomology we obtain a short exact sequence \begin{equation}0\To \Omega^{j-1}_{R/pR}\xto{\rho_j^1}H^0_\sub{\'et}(\Spec(R/pR),\frak i^*R^j\frak j_*\mu_p^{\otimes j})\xto{\bor\oplus\bor_p}\Omega^{j-1}_{R/pR,\sub{log}}\oplus \Omega^j_{R/pR,\sub{log}}\To 0\label{eqn1}\end{equation} and an isomorphism \begin{equation}\bor\oplus\bor_p:H^1_\sub{\'et}(\Spec(R/pR), \frak i^*R^j\frak j_*\mu_p^{\otimes j})\isoto \tilde\nu(j-1)(R/pR)\oplus\tilde\nu(j)(R/pR)\label{eqn2}\end{equation} for each $j\ge0$.

First we treat the case that $p$ is odd. Then we have the following commutative diagram
\[\xymatrix@R=5mm{
0\ar[r] & \Omega^{j-1}_{R/pR}\ar[r]^-{\rho_j^1} & H^0_\sub{\'et}(\Spec(R/pR), \frak i^*R^j\frak j_*\mu_p^{\otimes j})\ar[r]^-{\bor\oplus\bor_p} & \Omega^{j-1}_{R/pR,\sub{log}}\oplus \Omega^j_{R/pR,\sub{log}}\ar[r] & 0 \\
&&H^j_\sub{\'et}(R[\tfrac1p],\mu_p^{\otimes j})\ar[u]^{\al} && \\
&& k_j^M(R[\tfrac1p])\ar[u]^{h_1^j} && \\
 & \Omega^{j-1}_{R/pR}\ar@{=}[uuu] \ar[r]^{\rho_j^1} & k_j^M(R)\ar[u] \ar[r] & k_j^M(R/pR)\ar[r]\ar[uuu]_{(0,\sub{dlog})} & 0 
}\] 
in which the top row is exact by (\ref{eqn1}) and the bottom row by Corollary \ref{corollary_graded_pieces}(i)). So we immediately deduce that the map $\rho_j^1$ in the bottom row is injective. Next, the map $\al$ is surjective: indeed, it fits into a short exact sequence \begin{equation}0\To H^1_\sub{\'et}(\Spec(R/pR),\frak i^*R^{j-1}\frak j_*\mu_p^{\otimes j})\xto{\beta} H^j_\sub{\'et}(R[\tfrac1p],\mu_p^{\otimes j})\xto{\al} H^0_\sub{\'et}(\Spec(R/pR), \frak i^*R^j\frak j_*\mu_p^{\otimes j})\to 0.\label{eqn3}\end{equation} This follows from the spectral sequence 
\[E_2^{s,t}=H^s_\sub{\'et}(\Spec(R/pR),\frak i^*R^t\frak j_*(\mu_p^{\otimes j}))\Rightarrow H^{s+t}_\sub{\'et}(R[\tfrac1p],\mu_p^{\otimes j})\]
and the fact that $\Spec(R/pR)$ has \'etale cohomological dimensional $\le 1$ for $p$-torsion \'etale sheaves (we remark that (\ref{eqn3}) remains valid when $p=2$ and will be used later in that case).

The composition $\al\circ h_1^j$ is an isomorphism: indeed, there is an exact sequence \[\Omega^{j-1}_{R/pR}\To k_j^M(R[\tfrac1p])\xto{\bor\oplus\bor_p}k_{j-1}^M(R/pR)\oplus k_j^M(R/pR)\To 0\] by splicing Lemma \ref{lemma_localisation_in MilnorK} (with $t=p$) and Corollary \ref{corollary_graded_pieces}(i), and this may be compared to the top row of the diagram to prove the isomorphism.

The hypotheses of the proposition, and hence the diagram, remain valid if we replace $R$ by $\roi_F:=$ the henselisation of the discrete valuation ring $R_{pR}$; as indicated, we write $F:=\Frac\roi_F=\roi_F[\tfrac1p]$. Then the cohomological symbol $h^j_1:k_j^M(F)\isoto H^j_\sub{\'et}(F,\mu_p^{\otimes j})$ is an isomorphism by Bloch--Kato \cite[Thm.~5.12]{Bloch1986}; combined with the isomorphism of the previous paragraph (for $\roi_F$ rather than $R$), we deduce that $\al_{\roi_F}:H^j_\sub{\'et}(F,\mu_p^{\otimes j})\to H^0_\sub{\'et}(\Spec(\roi_F/p\roi_F), \frak i^*R^j\frak j_*\mu_p^{\otimes j})$ is an isomorphism.

Two results of Gabber, which we recall in Theorem \ref{theorem_Gabber_injectivity}, imply that $H^j_\sub{\'et}(R[\tfrac1p],\mu_p^{\otimes j})\to H^j_\sub{\'et}(F,\mu_p^{\otimes j})$ is injective. So from the injectivity of $\al_{\roi_F}$ and the surjectivity of $\al$ (and their compatibility), it follows that $\al$ is an isomorphism. Therefore $h_1^j$ is an isomorphism and a diagram chase reveals that $k_j^M(R)\to k_j^M(R[\tfrac1p])$ is injective.
Diagrammatically, this last argument can be summarised as follows:
\[\xymatrix@R=5mm{
k_j^M(F)\ar[r]^-\cong_-{h^j_1} & H^j_\sub{\'et}(F,\mu_p^{\otimes j})\ar[r]^-\cong_-{\alpha_{\roi_F}} & H^0_\sub{\'et}(\Spec(\roi_F/p\roi_F), \frak i^*R^j\frak j_*\mu_p^{\otimes j}) \\
 k_j^M(R[\tfrac1p]) \ar[r]_-{h^j_1} & H^j_\sub{\'et}(R[\tfrac1p],\mu_p^{\otimes j}) \ar@{^{(}->}[u] \ar@{->>}[r]_-\al & H^0_\sub{\'et}(\Spec(R/pR), \frak i^*R^j\frak j_*\mu_p^{\otimes j}). \ar[u]
}\] 

It remains to treat the case $p=2$. In that case we have a diagram (ignoring the dashed arrow for the moment)
\[\xymatrix@R=5mm{
0\ar[r] & \Omega^{j-1}_{R/pR}\ar[r]^-{\rho_j^1} & H^0_\sub{\'et}(\Spec(R/pR), \frak i^*R^j\frak j_*\mu_p^{\otimes j})\ar[r]^{\bor\oplus\bor_p} & \Omega^{j-1}_{R/pR,\sub{log}}\oplus \Omega^j_{R/pR,\sub{log}}\ar[r] & 0 \\
&&H^j_\sub{\'et}(R[\tfrac1p],\mu_p^{\otimes j})\ar[u]^{\al} && \\
&k_j^M(R[\tfrac1p])\ar[ur]^{h_1^j}& \tilde\nu(j-1)(R/pR)\oplus\tilde\nu(j-2)(R/pR)\ar[d]^{\rho^2_j\oplus\lrho^2_j}\ar[u]_{\beta\circ\gamma} && \\
&&k_j^M(R)\ar[ul]\ar[d]&&\\
 & \Omega^{j-1}_{R/pR}\ar@/^20mm/[uuuu]^= \ar[r]^{\rho_j^1} & \op{coker}(\rho^2_j\oplus\lrho^2_j) \ar[r]\ar@/_30mm/@{-->}[uuuu] & k_j^M(R/pR)\ar[r]\ar[uuuu]_{(0,\sub{dlog})} & 0 
}\] 
in which the top row is exact by (\ref{eqn1}) and the bottom row by Corollary \ref{corollary_graded_pieces}(ii). The map $\rho^2_j\oplus\lrho^2_j$ is that of Corollary \ref{corollary_graded_pieces}(ii). The top half of the middle of the diagram consists of the exact sequence 
\[0\To \tilde\nu(j-1)(R/pR)\oplus\tilde\nu(j-2)(R/p)\xto{\beta\circ\gamma} H^j_\sub{\'et}(R[\tfrac1p],\mu_p^{\otimes j})\xto{\al} H^0_\sub{\'et}(\Spec(R/pR), \frak i^*R^j\frak j_*\mu_p^{\otimes j})\To 0\] obtained by combining (\ref{eqn3}) with the isomorphism
\[\gamma:\tilde\nu(j-1)(R/pR)\oplus\tilde\nu(j-2)(R/pR)\stackrel{\sub{(\ref{eqn2})}}\cong H^1_\sub{\'et}(\Spec(R/pR), \frak i^*R^{j-1}\frak j_*\mu_p^{\otimes j-1})\stackrel{\zeta_p}\cong H^1_\sub{\'et}(\Spec(R/pR), \frak i^*R^{j-1}\frak j_*\mu_p^{\otimes j}),\] where we use (\ref{eqn2}) for $j-1$ with the order of the two summands swapped, and where $\zeta_p$ denotes the isomorphism given by cupping with the primitive $p^\sub{th}$-root of unity $\zeta_p=-1\in R$.

It is clear that the diagram commutes, except for identifying the two maps \[\tilde\nu(j-2)(R/pR)\oplus\tilde\nu(j-1)(R/pR)\to H^j_\sub{\'et}(R[\tfrac1p],\mu_p^{\otimes j}).\] But by functoriality and Gabber's aforementioned injectivity $H^j_\sub{\'et}(R[\tfrac1p],\mu_p^{\otimes j})\into H^j_\sub{\'et}(F,\mu_p^{\otimes j})$ this reduces to the analogous commutativity for $\roi_F$ in place of $R$; in that case the commutativity is implicit in \cite[(5.15)]{Bloch1986} and mentioned explicitly in \cite[\S4.4.8]{ColliotThelene1999}.

The commutativity of the diagram has two consequences. Firstly, it implies that $\rho^2_j\oplus\lrho^2_j$ is injective. Secondly, it implies that the map $k_j^M(R)\to H^0_\sub{\'et}(\Spec(R/pR), \frak i^*R^j\frak j_*\mu_p^{\otimes j})$ factors through $\op{coker}(\rho^2_j\oplus\lrho^2_j)$; the dashed map therefore exists and the diagram continues to commute.

With the dashed map in place, it follows immediately from a simple diagram chase that the $\rho_j^1$ in the bottom row is injective. A slightly longer diagram chase shows that $k_j^M(R)\to k_j^M(R[\tfrac1p])$ is injective.

The map $\bor\oplus\bor_p:k_j^M(R[\tfrac1p])\to \Omega^{j-1}_{R/pR,\sub{log}}\oplus\Omega^j_{R/pR,\sub{log}}$ is surjective and its kernel is contained in the image of $k_j^M(R)\to k_j^M(R[\tfrac1p])$ by Lemma \ref{lemma_localisation_in MilnorK}. Diagram chasing now shows that $h_1^j$ is an isomorphism.
\end{proof}

\section{\texorpdfstring{$p$}{p}-adic Nesterenko--Suslin isomorphism}
The Nesterenko--Suslin isomorphism \cite{Suslin1989}, reproved by Totaro \cite{Totaro1992}, states that for any field $k$ there is a natural isomorphism $K_j^M(k)\isoto H^j(k,\bb Z(j))$ where the target denotes the weight $j$, degree $j$ motivic cohomology of $k$. This was extended by Kerz to all regular local rings containing a field \cite{Kerz2009} \cite[Prop.~10(11)]{Kerz2010} (replacing Milnor $K$-theory by improved Milnor $K$-theory in the small residue field case).

In this section we prove the following $p$-adic analogue in mixed characteristic; the generality in which $\bb Z_p(j)(A)$ is defined allows us to avoid any regularity hypotheses:

\begin{theorem}\label{theorem_Nesterenko_Suslin}
The Galois cohomological symbol induces, for any local, $p$-henselian ring $A$, natural isomorphisms \[\hat K_j^M(A)/p^r\isoto H^j(\bb Z/p^r(j)(A))\] for all $r,j\ge0$.
\end{theorem}
\begin{proof}
We begin by explaining the existence of a natural symbol map $K_j^M(A)\to H^j(\bb Z/p^r(j)(A))$, induced by the Galois symbol in the smooth case. So suppose first that $A$ is a local, $p$-henselian, ind-smooth $\bb Z_{(p)}$-algebra; to put ourselves in the context of our main theorems (which are stated over complete discrete valuation rings), let $A'$ be the $p$-henselisation of $A\otimes_{\bb Z_{(p)}}\bb Z_p$ and note that $A'$ is a local, $p$-henselian, ind-smooth $\bb Z_p$-algebra. The Galois symbol therefore induces \[\hspace{-3mm}K_j^M(A)\to K_j^M(A')\to \ker(H^j_\sub{\'et}(A'[\tfrac1p],\mu_{p^r}^{\otimes j})\to W_r\Omega_{A'/p A',\sub{log}}^{j-1}))\stackrel{\sub{Corol.~\ref{corollary_BCM}}}\cong H^j(\bb Z_p(j)(A')/p^r)\stackrel{\simeq}{\leftarrow}H^j(\bb Z_p(j)(A)/p^r),\] where the final isomorphism holds since $A$ and $A'$ have the same $p$-adic completions (from which it also follows that the first arrow is an isomorphism modulo any power of $p$, which will be relevant later). This defines the desired map $K_j^M(A)\to H^j(\bb Z_p(j)(A)/p^r)$ whenever $A$ is a local, $p$-henselian, ind-smooth $\bb Z_{(p)}$-algebra.

For general local, $p$-henselian rings $A$, we define $K_j^M(A)\to H^j(\bb Z_p(j)(A)/p^j)$ via left Kan extension from the ind-smooth case, using Propositions \ref{proposition_motivic_cohomol_LKE} and \ref{proposition_LKE_Milnor}. Concretely, this means that if we pick any local, $p$-henselian, ind-smooth $\bb Z_{(p)}$-algebra $R$ surjecting onto $A$, then the map $K_j^M(R)\to H^j(\bb Z_p(j)(R)/p^r)$ defined in the previous paragraph descends (necessarily uniquely) to a map $K_j^M(A)\to H^j(\bb Z_p(j)(A)/p^r)$. Our goal is to show that this descends further to an isomorphism $\hat K_j^M(A)/p^r\isoto H^j(\bb Z_p(j)(A)/p^r)$.

We have already established this isomorphism in two cases:
\begin{enumerate}
\item ``smooth case'', namely whenever $A=R$ is a local, ind-smooth $p$-henselian algebra over a complete discrete valuation ring: indeed, this is precisely Theorem \ref{theorem1};
\item and ``big residue field case'', namely whenever $A$ is a local, $p$-henselian algebra over a complete discrete valuation ring with big residue field: this is precisely line (\ref{eqn_NS}) of the proof of Proposition~\ref{proposition_reduction_unram}, which was conditional at the time on the now-established Theorem \ref{theorem1}.
\end{enumerate}
To establish the isomorphism in general we may reduce, by taking a filtered colimit, to the case that $A$ is the $p$-henselisation of a local, essentially finite type $\bb Z_{(p)}$-algebra. Applying the same trick as in the proof of Proposition \ref{prop_reduction1}, we realise the residue field $K$ of $A$ as a finite extension of $\bb F_p(t_1,\dots,t_d)$ and then pick a big enough finite field $k'$ such that $|k':\bb F_p|$ is coprime to both $p$ and $|K:\bb F_p(\ul t)|$. Let $A'$ be the $p$-henselisation of $A\otimes_{\bb Z_{(p)}}W(k')$, which is a local, $p$-henselian, $W(k')$-algebra, and consider the commutative diagram
\[\xymatrix{
K_j^M(A')/p^r\ar[r]^-\cong & H^j(\bb Z_p(j)(A')/p^r)\ar@{}[r]|{\cong}&H^j(\bb Z_p(j)(\hat{A'})/p^r)\\
K_j^M(A)/p^r\ar[r]\ar[u] & H^j(\bb Z_p(j)(A)/p^r)\ar[u]\ar@{}[r]|{\cong}&H^j(\bb Z_p(j)(\hat A)/p^r)\ar[u]
}\]
The symbol map in the top row is an isomorphism by the big residue field case. The right vertical arrow is injective since $\hat{A'}=\hat A\otimes_{\bb Z_p}W(k')$ is a finite \'etale extension of $\hat A$ of degree prime to $p$: see Corollary~\ref{corollary_injective}. Since $K_j^M(A')/p^r\isoto\hat K_j^M(A')/p^r\supseteq \hat K_j^M(A)/p^r$, it now follows formally from a diagram chase that the symbol map for $A$ descends to an injection $\hat K_j^M(A)/p^r\into H^j(\bb Z_p(j)(A)/p^r)$.

To prove surjectivity, pick an $p$-henselian, ind-smooth $\bb Z_{(p)}$-algebra surjecting onto $A$ and henselise it along the kernel; the result $R$ is an $p$-henselian, ind-smooth $\bb Z_p$-algebra equipped with a henselian surjection $R\to A$. As at the end of the proof of Proposition \ref{proposition_motivic_cohomol_LKE}, the induced map on completions $\hat R\to \hat A$ is still a henselian surjection, and so rigidity for the $\bb Z_p(j)$ \cite[Thm.~5.2]{AntieauMathewMorrowNikolaus} implies that $H^j(\bb Z_p(j)(R)/p^r)\to H^j(\bb Z_p(j)(A)/p^r)$ is surjective. This reduces the necessary surjectivity to the case of $R$ in place of $A$; but that is covered by the smooth case mentioned above.
\end{proof}

\begin{remark}
Theorem \ref{theorem_Nesterenko_Suslin} implies, in particular, that the maps $H^j(\bb Z/p^{r+1}(j)(A))\to H^j(\bb Z/p^r(j)(A))$ are surjective for all $r\ge 1$; equivalently $H^{j+1}(\bb Z_p(j)(A))$ is $p$-torsion-free, whence $H^{j}(\bb Z_p(j)(A))/p^r\isoto H^{j}(\bb Z/p^r(j)(A))$ for all $r\ge1$.

We record here an independent proof of these results bypassing our main theorems. We wish to show that $p:H^{j+1}(\bb Z_p(j)(R)/p^r)\to H^{j+1}(\bb Z_p(j)(R)/p^{r+1})$ is injective; as at the end of the proof of the previous theorem we may use rigidity to reduce to the case that $A=R$ is a local, $p$-henselian, ind-smooth $\bb Z_p$-algebra, then we use rigidity again to replace $R$ by $\res R:=R/pR$.

The exact sequence $0\to W_{r-1}\Omega_\sub{log}^j\xto{p}W_{r}\Omega_\sub{log}^j\to \Omega_\sub{log}^j\to 0$ of \'etale sheaves on $\Spec \res R$ induces upon taking cohomology a sequence \[0\to W_{r-1}\Omega_{\res R,\sub{log}}^j\xto{p}W_{r}\Omega_{\res R,\sub{log}}^j\to \Omega_{\res R,\sub{log}}^j\to\tilde\nu_{r-1}(j)(\res R)\xto{p}\tilde\nu_r(j)(\res R)\to\tilde\nu(j)(\res R)\to 0\] where $\tilde\nu_r(j)(\res R):=H^1_\sub{\'et}(\Spec\res R,W_r\Omega^j_\sub{log})$.

But the projection map $W_{r}\Omega_{\res R,\sub{log}}^j\to \Omega_{\res R,\sub{log}}^j$ is surjective since both sides are quotients of $K_j^M(\res R)$ (see Remark \ref{remark_WrOmegalog}), so we have proved exactness of $0\to\tilde\nu_{r-1}(j)(\res R)\xto{p}\tilde\nu_r(j)(\res R)\to\tilde\nu(j)(A)\to 0$. Recalling that $H^{j+1}(\bb Z/p^r(j)(A))=\tilde\nu_r(j)(A)$ \cite[Corol.~8.19]{BhattMorrowScholze2}, this is precisely the desired result.
\end{remark}

To treat the passage from Milnor $K$-theory to its improved variant, we unsurprisingly needed norm maps on $\bb Z_p(j)(-)$. Although these should exist in much greater generality (say, for arbitrary finite quasisyntomic maps), the case of finite \'etale maps is sufficient for our purposes, namely Corollary \ref{corollary_injective}.

\begin{lemma}\label{lemma_norm_exists}
To any finite \'etale map $A\to B$ of $p$-adically complete rings there are associated norm maps $N=N_{B/A}:\bb Z_p(j)(B)\to \bb Z_p(j)(A)$ for all $j\ge0$ satisfying the following properties:
\begin{enumerate}
\item For any map $A\to A'$ where $A'$ is another $p$-adically complete ring, the diagram
\[\xymatrix{
\bb Z_p(j)(A) \ar[d] &  \ar[l]_N \bb Z_p(j)(B)\ar[d]\\
\bb Z_p(j)(A') & \ar[l]_N \bb Z_p(j)(A'\otimes_AB)
}\]
commutes.
\item When $A\to B$ has constant degree $d$, then the composition $\bb Z_p(j)(A)\to \bb Z_p(j)(B)\xto{N}\bb Z_p(j)(A)$ is multiplication by $d$.
\end{enumerate}

\end{lemma}
\begin{proof}
Suppose first that $A$ is quasisyntomic, in which case $B$ is automatically also quasisyntomic. Letting $A\to S$ be a quasisyntomic cover where $S$ is quasiregular semiperfectoid, then quasisyntomic descent of $\bb Z_p(j)$ implies that $\bb Z_p(j)(A)$ is equivalent to the totalisation of $\bb Z_p(j)(-)$ of the $p$-completed \v Cech nerve of this cover, i.e., \[\bb Z_p(j)(A)\simeq|\cosimp{S}{S\hat\otimes_AS}{S\hat\otimes_AS\hat\otimes_AS} |\] The base change $S_B:=S\otimes_AB$ is a quasiregular semiperfectoid ring providing a quasisyntomic cover of $B$, and so similarly $\bb Z_p(j)(B)\simeq|\bb Z_p(j)(S_B^{\hat\otimes_B\blob})|$.

Since all terms in the \v Cech nerves are themselves quasiregular semiperfectoid rings \cite{}, there are natural equivalences $\bb Z_p(j)(S^{\hat\otimes_An})\simeq\tau_{[2j-1,2j]}K(S^{\otimes_An};\bb Z_p)$ for all $n\ge0$, and similarly for $S_B$, and therefore the $K$-theory norm map for the finite \'etale extension $S^{\hat\otimes_An}\to S_B^{\hat\otimes_Bn}=S^{\hat\otimes_An}\otimes_AB$ induces $N_{S_B^{\hat\otimes_Bn}/S^{\hat\otimes_An}}:\bb Z_p(j)(S_B^{\hat\otimes_Bn})\to \bb Z_p(j)(S^{\hat\otimes_An})$. These are moreover compatible in the sense that the diagram
\[\xymatrix{
\bb Z_p(j)(S^{\hat\otimes_An}) \ar[d]_f &  \ar[l]_N \bb Z_p(j)(S_B^{\hat\otimes_Bn})\ar[d]^f\\
\bb Z_p(j)(S^{\hat\otimes_Am}) & \ar[l]_N \bb Z_p(j)(S_B^{\hat\otimes_Bm})
}\]
commutes for any map $f:[m]\to [n]$, since the analogous diagram commutes in $K$-theory; so we may totalise to induce a norm map $N_{B/A}^S:\bb Z_p(j)(B)\to \bb Z_p(j)(A)$ which appears to depend on $S$.

But this dependence on $S$ is superficial: given any quasisymtomic cover of $A$ by another quasiregular semiperfectoid $T$, which we assume receives a map from $S$ since otherwise we replace it by $S\otimes_AT$, the compatible commutative diagrams
\begin{equation}\xymatrix{
\bb Z_p(j)(S^{\hat\otimes_An}) \ar[d] &  \ar[l]_N \bb Z_p(j)(S_B^{\hat\otimes_Bn})\ar[d]\\
\bb Z_p(j)(T^{\hat\otimes_An}) & \ar[l]_N \bb Z_p(j)(T_B^{\hat\otimes_Bn})
}\label{eqn_norm_1}\end{equation}
(again, since the analogous diagrams in $K$-theory commute and are compatible over $n$) totalise to a commutative diagram
\[\xymatrix{
\bb Z_p(j)(A)\ar[d]^{\rotatebox{90}{$\simeq$}} & \bb Z_p(j)(B)\ar[d]^{\rotatebox{90}{$\simeq$}}\\
|\bb Z_p(j)(S^{\hat\otimes_A\blob})| \ar[d]^{\rotatebox{90}{$\simeq$}} &  \ar[l]_N {|\bb Z_p(j)(S_B^{\hat\otimes_B\blob})|}\ar[d]^{\rotatebox{90}{$\simeq$}}\\
|\bb Z_p(j)(T^{\hat\otimes_A\blob})| & \ar[l]_N {|\bb Z_p(j)(T_B^{\hat\otimes_B\blob})|}
}\]
i.e., informally $N_{B/A}^S=N_{B/A}^{S'}$. More precisely, we therefore define $N_{B/A}$ be be the limit of $N_{B/A}^S$ over all covers $S$ of $A$, viewing $\bb Z_p(A)$ as the limit of $|\bb Z_p(j)(S^{\hat\otimes_A\blob})|$ over all such covers. But in practice in what follows, we will just pick a particular cover $S$.

We next check functoriality (i), assuming that both $A$ and $A'$ are quasisyntomic so that the norm maps have been defined. Let $A\to S$ be a quasisyntomic cover with $S$ quasiregular semiperfectoid, and then let $S\hat\otimes_AA'\to T$ a quasisyntomic cover with $T$ quasiregular semiperfectoid; note that the composition $A'\to T$ is a quasisyntomic cover. Calculating the norms for $A\to B$ and $A'\to A'\otimes_AB$ using the covers $S$ and $T$ respectively, the desired functoriality follows from totalising over the commutative diagrams which are almost identical to (\ref{eqn_norm_1}), just replacing $A$ and $B$ in the bottom row by $A'$ and $B'=A'\otimes_AB$.

We next extend the norm map to finite \'etale maps $A\to B$ between arbitrary $p$-complete rings; the idea is to compatibly simpicialy resolve $A$ and $B$ to reduce to the case already treated. Similarly to the proof of Proposition \ref{proposition_LKE_Milnor} we pick a simplicial resolution $R_\blob\to A$ by ind-smooth $\bb Z_p$-algebras such that the kernel of each surjection $R_q\to A$ is a henselian ideal. Then $R_\blob/p^sR_\blob\quis A\dotimes_{\bb Z}\bb Z/p^s\bb Z$ for each $s\ge1$, whence taking the derived inverse limit shows that $\hat R_\blob\quis A$; here $\hat R_\blob$ is the simplicial ring obtained by $p$-adically completing each $R_q$ (which represents the derived $p$-adic completion of $R_\blob$), and we recall that $A$ is derived $p$-adically complete (as it is $p$-adically complete, which includes separated under our conventions). The kernel of each surjection $\hat R_q\to A$ is henselian, as follows from a straightforward series of manipulations which we leave to the reader.

In conclusion, we have constructed a simplicial resolution $\hat R_\blob\quis A$ along henselian surjections, where each ring in the resolution is quasisyntomic. The finite \'etale map $A\to B$ lifts uniquely to a finite \'etale map $\hat R_q\to Q_q$ for each $q\ge0$, and these assemble to form a simplicial ring $Q_\blob\to B$. Observe that each $Q_q$ is also quasisyntomic over $\bb Z_p$, that $Q_q\to B$ is a henselian surjection (as henselian surjections are preserved under base change along integral maps), and that $Q_\blob\to B$ is an equivalence (if $B$ were a finite free $A$-module then $Q_\blob$ would be a finite free $\hat R_\blob$-module and this would be clear; the general case is a direct summand of such a free case).

The already constructed norm maps in the quasisyntomic case therefore define a map $N:\bb Z_p(j)(\hat R_\blob)\to \bb Z_p(j)(Q_\blob)$ of simplicial complexes. Since $\bb Z_p(j)(-)$ commutes with $p$-completed sifted colimits \cite[Thm.~5.1(2)]{AntieauMathewMorrowNikolaus}, we may then geometrically realise and $p$-complete to define $N_{B/A}:\bb Z_p(j)(B)\to \bb Z_p(j)(A)$.

The independence of $N_{B/A}$ on the chosen resolution is proved similarly to the independence on $S$ in the first part of the proof, as is its functoriality (i); since we do not need these results for Corollary \ref{corollary_injective}, we omit the details of the proofs.

Property (ii) reduces via the definitions to the case that $A$ and $B$ are quasiregular semiperfectoid, in which case it follows from the analogous property of the $K$-theory norm map.
\end{proof}

\begin{corollary}\label{corollary_injective}
Let $A\to B$ be a finite \'etale map of $p$-adically complete rings, of constant degree not divisible by $p$. Then $\bb Z_p(j)(A)\to\bb Z_p(j)(B)$ is split injective.
\end{corollary}
\begin{proof}
Letting $d$ be the degree, the splitting is provided by $\tfrac1dN_{B/A}$ thanks to Lemma \ref{lemma_norm_exists}(ii).
\end{proof}

As an application of the $p$-adic Nesterenko--Suslin theorem, we describe the $p$-adic Milnor K-groups locally on smooth formal schemes over the rings of integers of perfectoid fields:

\begin{theorem}\label{theorem_Milnor_K_groups_perfectoids}
Let $C$ be a perfectoid field of characteristic $0$ containing all $p$-power roots of unity, and $\frak X$ a smooth, $p$-adic formal $\roi_C$-scheme; let $x\in \frak X$ and let $R:=\roi_{X,x}$ be the corresponding local ring; let $j\ge0$.
\begin{enumerate}
\item The Galois symbol \[\hat K_j^M(R)/p^r\To H^j_\sub{\'et}(R[\tfrac1p],\mu_{p^r}^{\otimes j})\] is an isomorphism; if $R$ has big residue field then the canonical map $K_j^M(R)/p^r\to K_j^M(R[\tfrac1p])/p^r$ is also an isomorphism.
\item The $p$-adic completion of $\hat K_j^M(R)$ is $p$-power-free.
\end{enumerate}
\end{theorem}
\begin{proof}
We begin by recalling the structure of the ring $R$. Picking an affine open neighbourhood $\Spf S$ of $x$, with $x$ corresponding to the prime ideal $\frak q\subseteq S$, then $R=\indlim_{f\in S\setminus\frak q}\hat{S[\tfrac1f]}$ where the hat denotes $p$-adic completion. In particular, $R$ is a $p$-henselian local ring having residue field $k(\frak q)$.

(i) The Galois symbol is an isomorphism thanks to Theorem \ref{theorem_Nesterenko_Suslin} and the known identification $H^j(\bb Z/p^r(j)(R))\cong H^j_\sub{\'et}(R[\tfrac1p],\mu_{p^r}^{\otimes j})$ \cite[Thm.~10.1]{BhattMorrowScholze2}. When $R$ has big residue field the Galois symbol factors through $\hat K_j^M(R)/p^r=K_j^M(R)/p^r\to K_j^M(R[\tfrac1p])/p^r$, so it remains only to prove that this map is surjective: but that follows from the existence of $u,\pi\in \roi_C$, with $u$ being a unit, such that $p=u\pi^{p^r}$.

(ii) In the first diagram of the proof of Proposition \ref{prop_reduction0} the boundary map $\delta$ is injective since the previous map in the sequence $H^{j-1}(R[\tfrac1p],\mu_{p^r}^{\otimes j})\to H^{j-1}(R[\tfrac1p],\mu_{p}^{\otimes j})$ may be identified (using (i) for $j-1$ and trivalising the Tate twists) with the surjection $\hat K_j^M(R)/p^r\to\hat K_j^M(R)/p$. Again using (i) (this time for $j$), we deduce that the multiplication map $p:\hat K_j^M(R)/p^{r-1}\to \hat K_j^M(R)/p^r$ is injective, from which the $p$-torsion-freeness claim follows. \end{proof}

\begin{remark}
In \cite{Izhboldin1991} Izhboldin proves that the Milnor K-groups of a field of characteristic $p>0$ are $p$-torsion free. Combined with the Gersten conjecture for Milnor K-theory in equal characteristic, this implies that for a smooth scheme $X$ over a field $k$ of characteristic $p>0$ the (improved) Milnor K-sheaf $\hat{\cal K}^M_{j,X}$ is $p$-torsion free. Theorem \ref{theorem_Milnor_K_groups_perfectoids}(ii) may be considered to be an analogue of this statement in mixed characteristic. 
\end{remark}

\section{An application to the Gersten conjecture}\label{section_Gersten_conj}
The Gersten conjecture in Milnor $K$-theory predicts that, for any regular Noetherian local ring $R$, the Gersten complex \[0\To \hat K_j^M(R)\To K_j^M(\Frac R)\To \bigoplus_{x\in \Spec R^{(1)}}K_{j-1}^M(k(x))\To \bigoplus_{x\in \Spec R^{(2)}}K_{j-2}^M(k(x))\To\cdots\] is exact for each $j\ge0$; here we write $\Spec R^{(i)}$ for the set of codimension $i$ points of $\Spec R$. If $R$ contains a field then the Gersten complex is known to be universally exact by Kerz \cite{Kerz2009, Kerz2010}, but the mixed characteristic case is open.

\subsection{The \texorpdfstring{$p$}{p}-henselian, ind-smooth case}\label{ss_Gersten}
By combining Theorem \ref{theorem_BK}(ii) with Kerz' and other existing Gersten results in motivic cohomology, we prove the Gersten conjecture in mod $p$-power Milnor $K$-theory for $p$-henselian, ind-smooth algebras over complete discrete valuation rings:

\begin{theorem}\label{theorem_Gersten}
Let $V$ be a complete discrete valuation ring of mixed characteristic, and $R$ a $p$-henselian, regular, Noetherian, local $V$-algebra such that $R/\frak mR$ is ind-smooth over $V/\frak m$. Then the mod $p$-power Gersten conjecture holds for $R$, i.e., for any $r,j\ge0$, the complex \[0\To \hat K_j^M(R)/p^r\To K_j^M(\Frac R)/p^r\To \bigoplus_{x\in \Spec R^{(1)}}K_{j-1}^M(k(x))/p^r\To \bigoplus_{x\in \Spec R^{(2)}}K_{j-2}^M(k(x))/p^r\To\cdots\] is exact.
\end{theorem}
\begin{proof}
As mentioned before the theorem, this will follow from combining our main injectivity theorem with existing results; there are various ways to carry out the argument, among which we propose the following (whose advantage is that it circumvents any new use of Panin's trick \cite{Panin2003} to reduce the ind-smooth case to the smooth case). First note that the hypotheses imply that $V\to R$ is geometrically regular, therefore ind-smooth by N\'eron--Popescu; so our earlier results do apply.

Let $Z=\Spec(R/\frak mR)$ and $X_{\eta}=\Spec(R[\tfrac1p])$ be the special and generic fibres, so that we have Gersten complexes
\[g_j(X)=\qquad0\To K^M_j(\Frac R)/p^r
\rightarrow \bigoplus_{x\in X^{(1)}} K^M_{j-1}(x)/p^r\rightarrow\cdots\]
\[g_j(Z)=\qquad0\To K^M_j(\Frac R)/p^r
\rightarrow \bigoplus_{x\in Z^{(1)}} K^M_{j-1}(x)/p^r\rightarrow\cdots\]

\[g_j(X_\eta)=\qquad0\To K^M_j(\Frac R)/p^r
\rightarrow \bigoplus_{x\in X_\eta^{(1)}} K^M_{j-1}(x)/p^r\rightarrow\cdots.\]
fitting into a short exact sequence $0\to g_{j-1}(Z)[-1]\to g_j(X)\to g_j(X_\eta)\to 0$. Thanks to the Gersten conjecture in Milnor $K$-theory over fields \cite{Kerz2010}, the canonical map $\hat K_j^M(R)/p^r\to g_j(Z)$ is an equivalence and $g_j(X_\eta)$ calculates the Zariski cohomology of the improved Milnor $K$-theory sheaf $\hat{\cal K}_j^M/p^r$ on $X_\eta$ (which is typically not local).

From the short exact sequence of Gersten complexes we therefore obtain an exact sequence
\begin{equation}0\To H^0(g_j(X))\To H^0_\sub{Zar}(X_\eta,\hat{\cal K}^M_j/p^r)\To\hat K^M_{j-1}(A/\frak mA)/p^r\To H^1(g_j(X))\To 0\label{eqn_Gersten}\end{equation}
and isomorphisms $H^n(g_j(X))\isoto H^n_\sub{Zar}(X_\eta,\hat{\cal K}_j^M/p^r)$ for $n\ge2$. But by passage to a filtered colimit over Lemma \ref{lemma_BL}(ii) below we know that \[H^n_\sub{Zar}(X_\pi,\hat{\cal K}_j^M/p^r)=\begin{cases} H^j_\sub{\'et}(X_\eta,\mu_{p^r}^{\otimes j}) & \text{if }n=0, \\ 0 & \text{if }n>0.\end{cases}\] This proves the desired acyclicity in degrees $\ge2$ and moreover lets us compare (\ref{eqn_Gersten}) to the short exact sequence of Theorem \ref{theorem1}, from which we immediately obtain the acyclicity in degree $1$ and the desired isomorphism $\hat K_j^M(R)/p^r\isoto H^0(g_j(X))$.
\end{proof}

We required the following (presumably well-known) result:

\begin{lemma}\label{lemma_BL}
Let $V$ be a complete discrete valuation ring of mixed characteristic, and $S$ an essentially smooth, local $V$-algebra; let $r,j\ge0$. Then
\begin{enumerate}
\item $H^n_\sub{Zar}(\Spec S[\tfrac1p],R^i\ep_*\mu_{p^r}^{\otimes j})=0$ for all $i\le j$ and all $n>0$, where $\ep:\Spec S[\tfrac1p]_\sub{\'et}\to\Spec S[\tfrac1p]_\sub{Zar}$ is the projection map of sites.
\item $H^n_\sub{Zar}(\Spec S[\tfrac1p],\hat{\cal K}_j^M/p^r)=\begin{cases} H^j_\sub{\'et}(\Spec S[\tfrac1p],\mu_{p^r}^{\otimes j}) & \text{if }n=0, \\ 0 & \text{if }n>0.\end{cases}$
\end{enumerate}
\end{lemma}
\begin{proof}
Part (ii) is follows from (i) via the change of topology spectral sequence, using the Bloch--Kato isomorphism $\hat{\cal K}_j^M/p^r\isoto R^j\ep_*\mu_{p^r}^{\otimes j}$ on $\Spec S[\tfrac1p]$.

To prove part (i) we use known Gersten results in motivic cohomology.  For any essentially finite type scheme $Y$ over $V$, let $\bb Z(j)^\sub{mot}(Y):=z^j(Y,\blob)[-2j]$ denote its motivic cohomology complex defined via Bloch's higher Chow groups. Similarly to the previous proof above with Milnor $K$-theory, there are associated Gersten complexes for any $i,j\ge0$ \cite{Geisser2004}
\[g_i(j)^\sub{mot}(\Spec S)=\qquad 0\To H^i(\bb Z(j)^\sub{mot}(\Frac S)/p^r)\To\bigoplus_{x\in \Spec S^{(1)}}H^{i-1}(\bb Z(j-1)^\sub{mot}(k(x))/p^r)\To \cdots\]
\[g_i(j)^\sub{mot}(\Spec S/\frak mS )=\qquad 0\To H^i(\bb Z(j)^\sub{mot}(\Frac(S/\frak mS))/p^r)\To\bigoplus_{x\in \Spec S/\frak mS ^{(1)}}H^{i-1}(\bb Z(j-1)^\sub{mot}(k(x))/p^r)\To \cdots\]
\[g_i(j)^\sub{mot}(\Spec S[\tfrac1p])=\qquad 0\To H^i(\bb Z(j)^\sub{mot}(\Frac(S))/p^r)\To\bigoplus_{x\in \Spec S[\tfrac1p]^{(1)}}H^{i-1}(\bb Z(j-1)^\sub{mot}(k(x))/p^r)\To \cdots\]
fitting into short exact sequences \[0\to g_{i-1}(j)^\sub{mot}(\Spec S/\frak mS)[-1]\to g_i(j)^\sub{mot}(\Spec S)\to g_i(j)^\sub{mot}(\Spec S[\tfrac1p])\to 0.\] However, unlike the previous case of Milnor $K$-theory, we may now appeal to the fact that the Gersten conjecture in motivic cohomology is known not only for essentially smooth algebras over fields, but also for the essentially smooth $V$-algebra $S$ \cite[Corol.~4.5]{Geisser2004} (at least with mod $p^r$-coefficients). So from the corresponding long exact sequence we deduce that $H^n_\sub{Zar}(\Spec S[\tfrac1p],\cal H^i(\bb Z(j)^\sub{mot}/p^r))=0$ for $n>0$, where $\cal H^i(\bb Z(j)^\sub{mot}/p^r)$ denotes the Zariski sheafification of $U\mapsto H^i(\bb Z(j)^\sub{mot}(U)/p^r)$. The proof is completed by appealing to the Beilinson--Lichtenbaum isomorphism $\cal H^i(\bb Z(j)^\sub{mot}/p^r)\simeq R^i\ep_*\mu_{p^r}^{\otimes j}$ on $\Spec S[\tfrac1p]$ when $i\le j$.
\end{proof}

\begin{remark}
In the context of Theorem \ref{theorem_Gersten}, the injectivity at the beginning of the Gersten complex, namely $\hat K_j^M(R)/p^r\into K_j^M(\Frac R)/p^r$, can be deduced more directly. Indeed, letting $F$ be as in the statement of Theorem \ref{theorem_Gabber_injectivity}, it is enough to check that the composition \[\hat K_j^M(R)/p^r\To K_j^M(\Frac R)/p^r\To K_j^M(F)/p^r\xto{h_{p^r}^j}H^j_\sub{\'et}(F,\mu_{p^r}^{\otimes j})\] is injective. But this composition coincides with \[\hat K_j^M(R)/p^r\To H^j_\sub{\'et}(R[\tfrac1p],\mu_{p^r}^{\otimes j})\To H^j_\sub{\'et}(F,\mu_{p^r}^{\otimes j}),\] where the first arrow is injective by Theorem \ref{theorem1} and the second by Theorem \ref{theorem_Gabber_injectivity}.
\end{remark}

We note that, in particular, we have proved the mod $p$-power Gersten conjecture Nisnevich locally on smooth $V$-schemes:

\begin{corollary}\label{corollary_Gesten}
Let $V$ be a complete discrete valuation ring of mixed characteristic and $X$ a smooth $V$-scheme. Then the Gersten sequence of Nisnevich sheaves on $X$ \[0\To\hat{\cal K}_{j,X}^M/p^r\To\bigoplus_{x\in X^{(0)}}i_{x*}(K_j^M(x)/p^r)\To \bigoplus_{x\in X^{(1)}}i_{x*}(K_{j-1}^M(x)/p^r)\To\cdots \] (which will be explained further in the course of the proof) is exact, and consequently there is a natural Bloch--Quillen isomorphism \[\CH^j(X)/p^r\cong H^j_\sub{Nis}(X,\hat{\cal K}_{j,X}^M/p^r).\]
\end{corollary}
\begin{proof}
For each point $x\in X$, let $K_j^M(x)/p^r$ denote the Nisnevich sheaf on $\Spec k(x)$ defined by sending each \'etale $k(x)$-algebra $L$ to $K_j^M(L)/p^r$; recall that this is indeed a Nisnevich sheaf because it is additive \cite[1.2]{Nisnevich1989}. Let $i_x:\Spec k(x)_\sub{Nis}\to X_\sub{Nis}$ denote the canonical map of sites, and $i_{x_*}(K_j^M(x)/p^r)$ the resulting pushforward of this sheaf to $X_\sub{Nis}$.

In other words, for each \'etale $f:U\to X$ we have $i_{x_*}(K_j^M(x)/p^r)(U)=\bigoplus_{y\in f^{-1}(x)}K_j^M(k(y))/p^r$, and the Nisnevich Gersten complex \[\bigoplus_{x\in X^{(0)}}i_{x*}(K_j^M(x)/p^r)\To \bigoplus_{x\in X^{(1)}}i_{x*}(K_{j-1}^M(x)/p^r)\To\cdots \] is characterised by the fact that its restriction to $U_\sub{Zar}$ is the usual Zariski Gersten complex for all such $U$. Moreover, as sheaves on the site $\Spec k(x)_\sub{Nis}$ have no higher cohomology on any object, the same is true of their pushforwards to $X_\sub{Nis}$; in particular, each Nisnevich sheaf $i_{x_*}(K_j^M(x)/p^r)$ on $X$ has no higher cohomology. To prove that the aforementioned Nisnevich Gersten complex is a resolution of $\hat{\cal K}_{j,X}^M/p^r$ (which denotes the Nisnevich sheafification of $U\mapsto \hat K_j^M(\roi_U(U))/p^r$), we must check the exactness of the Gersten complexes \[0\To\hat{K}_{j}^M(A)/p^r\To\bigoplus_{x\in \Spec A^{(0)}}K_j^M(k(x))/p^r\To \bigoplus_{x\in \Spec A^{(1)}}K_{j-1}^M(k(x))/p^r\To\cdots \] where $A$ runs over the henselian local rings attached to all points of $X$. When the point lies in the generic fibre of $X$, so that $A$ contains a field, we appeal to Kerz \cite{Kerz2010}; when the point lies in the special fibre, so that $A$ is $p$-henselian, we instead appeal to Theorem \ref{theorem_Gersten}.

The Bloch--Quillen formula follows as the Gersten resolution allows us to calculate the cohomology of $\hat{\cal K}_{j,X}^M/p^r$ as \[H^j_\sub{Nis}(X,\hat{\cal K}_{j,X}^M/p^r)=\op{coker}\big(\bigoplus_{x\in X^{(j-1)}}k(x)^\times/p^r\to\bigoplus_{x\in X^{(j)}}\bb Z/p^r\big)=\CH^j(X)/p^r\qedhere\]
\end{proof}

\begin{remark}
With $V$ and $X$ as in Corollary \ref{corollary_Gesten}, the result implies by functoriality the existence of a natural restriction map $\CH^j(X)/p^r\to H^j_\sub{Nis}(X_s,\hat{\cal K}_{j,X_s}^M/p^r)$ for each $s\ge 1$, where $X_s:=X\otimes_VV/\frak m^s$ is the corresponding thickening of the special fibre. When $j=d$ is the relative dimension of $X$ then these restriction maps are surjective by earlier work of the first author \cite{Lueders2019, Lueders2020}, answering a question of Kerz--Esnault--Wittenberg \cite[Conj.~10.1]{EsnaultKerzWittenberg2016} in the smooth case.
\end{remark}

\subsection{Gersten for \texorpdfstring{$K_3^M/p^r$}{K3M} of (incomplete) discrete valuation rings}\label{subsection_KM3}
Using the relative Gersten arguments pioneered by Bloch \cite{Bloch1986c} and Gillet--Levine \cite{GilletLevine1987}, adapted to Milnor $K$-theory by the first author \cite{Lueders2020a}, (and Panin's trick \cite{Panin2003} to reduce the ind-smooth case to the essentially smooth case) the Gersten conjecture for ind-smooth algebras over a discrete valuation ring $V$ mostly reduces to case of the discrete valuation rings obtained by localising smooth $V$-algebras at the generic point of their special fibre. We therefore record here a case of the Gersten conjecture for discrete valuation rings which we were surprised not to find in the literature; it is independent from our main results and the style of argument is not new.

We first review some classical results; let $\roi$ be a discrete valuation ring of residue characteristic $p>0$, and denote its fraction field by $F$ and its residue field by $k$. Then the map $K_2(\roi)\to K_2(F)$ is injective by Dennis--Stein \cite{Dennis1975}, with $p$-torsion-free cokerel $k^\times$, whence $K_2(\roi)/p^r\to K_2(F)/p^r$ is also injective for any $r\ge 1$.

Suppose $\roi$ is equi-characteristic. Then the map $K_3(\roi)\to K_3(F)$ is injective by Quillen \cite{Quillen1973a}, and the cokernel $K_2(k)$ is $p$-torsion-free by Izhboldin \cite{Izhboldin1991} (identifying it with $K_2^M(k)$ by Matsumoto), whence $K_3(\roi)/p^r\to K_3(F)/p^r$ is again injective. Secondly, if $\roi$ has infinite residue field then $K_3^M(\roi)\to K_3(F)$ was shown to be injective by Suslin--Yarosh \cite{Suslin1991} (though this is of course superseded by Kerz \cite{Kerz2009}).

Suppose instead that $\roi$ has mixed characteristic $(0,p)$. Then the injectivity of $K_3(\roi)\to K_3(F)$ remains open, but that of $K_3(\roi;\bb Z/p^r\bb Z)\to K_3(F;\bb Z/p^r\bb Z)$ is a theorem of Geisser--Levine \cite[Thm. 8.2]{GeisserLevine2000}. If $p>2$ then $\hat K_3(\roi)/p^r$ injects into $K_3(\roi)/p^r\subseteq K_3(\roi;\bb Z/p^r\bb Z)$ \cite[Prop.~10(6)]{Kerz2009} and so we deduce that $\hat K_3^M(\roi)/p^r\to K_3^M(F)/p^r$ is also injective. Here we observe that the argument of Suslin--Yarosh may also be used to prove the latter injectivity, without the hypothesis that $p>2$:

\begin{proposition}
Let $\roi$ be a discrete valuation ring of mixed characteristic, with field of fractions $F$ and residue field $k$. Then there is an exact sequence \[0\To \hat K_3^M(\roi)/p^r\To K_3^M(F)/p^r\xto{\bor}K_2^M(k)/p^r\To 0\] for any $r\ge1$.
\end{proposition}
\begin{proof}
The canonical map $K_2(\roi,\pi\roi;\bb Z/p^r\bb Z)\to K_2(\roi;\bb Z/p^r\bb Z)$ is injective since $K_3(k;\bb Z/p^r\bb Z)$ is generated by symbols (which lift to $K_3(\roi;\bb Z/p^r\bb Z)$) by Geisser--Levine, and this restricts to an injection $K_2(\roi,\pi\roi)/p^r\to K_2(\roi)/p^r$. From the exact sequence $K_2(\roi,\pi\roi)\to  K_2(\roi)\to K_2(k)\to 0$ (which is exact on the right since $K_2(k)$ is generated by symbols, which lift) we therefore obtain a short exact sequence \[0\To K_2(\roi,\pi\roi)/p^r\To  K_2(\roi)/p^r\To K_2(k)/p^r\To 0.\]

Assuming that $\roi$ has infinite residue field, the proof is completed by recalling Suslin--Yarosh's argument, which is itself a $K_3^M$-version of Dennis--Stein's argument: taking \cite[Thms.~3.9 \& 4.1]{Suslin1991} modulo $p^r$ implies the existence of a cocartesian square
\[\xymatrix{
K_2(\roi,\pi\roi)/p^r\ar[r] \ar[d] & K_3^M(\roi)/p^r\ar[d] \\
K_2^M(\roi)/p^r\ar[r]_{\cdot\{\pi\}}&K_3^M(F)/p^r
}\]
which indeed completes the proof when combined with the above exact sequence and the identity $\hat K_2^M(\roi)=K_2(\roi)$.

When $\roi$ has finite residue field, the injectivity of $\hat K_3(\roi)/p^r\to K_3(F)/p^r$ then follows from a standard norm trick by picking a tower of finite \'etale extensions $\roi\subseteq\roi_1\subseteq\roi_2\subseteq\cdots$ of discrete valuation rings where each extension has degree prime to $p$.
\end{proof}

\begin{remark}
The focus of this subsection is in the incomplete case, but we mention that $\hat K_j^M(\roi)\to K_j^M(F)$ is known to be injective for all $j\ge0$ if $\roi$ is compelte and has finite residue field \cite{Dahlhausen2018}.
\end{remark}

\section{\texorpdfstring{$p$}{p}-adic Milnor $K$-theory of local \texorpdfstring{$\bb F_p$}{Fp}-algebras}
In this section we present two variants of the Bloch--Kato--Gabber theorem, describing the mod $p$-power Milnor $K$-theory of certain local $\bb F_p$-algebras. The analogous isomorphisms for algebraic $K$-theory are already known by earlier work of the second author and collaborators \cite[Thm.~5.27]{ClausenMathewMorrow2020} \cite[Thm.~2.1]{KellyMorrow2018a}; the new tool which allows us to treat Milnor $K$-theory is the left Kan extension observation of Proposition~\ref{proposition_LKE_Milnor}.

\subsection{Cartier smooth \texorpdfstring{$\bb F_p$}{Fp}-algebras}\label{subsection_KM_val_charp}
We start by recalling the following terminology from \cite{KellyMorrow2018a}: an $\bb F_p$-algebra $A$ is called {\em Cartier smooth} if it satisfies the following smoothness criteria:
\begin{itemize}\itemsep0pt
\item $\Omega^1_A$ is a flat $A$-module;
\item $H_j(\bb L_{A/\bb F_p})=0$ for all $j>0$;
\item the inverse Cartier map $C^{-1}: \Omega^j_A\to H^j(\Omega^\bullet_A)$ is an isomorphism for all $j\geq 0$.
\end{itemize}
Let $\op{CSm}_{\bb F_p}$ denote the category of Cartier smooth $\bb F_p$-algebras, which includes the category of smooth $\bb F_p$-algebras $\op{Sm}_{\bb F_p}$; to simplify notation in the following proof, let $\op{Sm}_{\bb F_p}^\Sigma$ be the subcategory of finitely generated polynomial $\bb F_p$-algebras.

As a more interesting example, results of Gabber--Ramero \cite[Thm.~6.5.8(ii) \& Corol.~6.5.21]{GabberRamero2003} and Gabber \cite[App.]{KerzStrunkTamme2018} state that any valuation ring of characteristic $p$ is Cartier smooth. 

The following description of the \'etale-syntomic cohomology of Cartier smooth algebras was implicit in \cite{KellyMorrow2018a}:

\begin{proposition}\label{proposition_LKE_Omegalog2}
Let $r,j\ge0$.
\begin{enumerate}
\item The functor $W_r\Omega^j_\sub{log}:\op{CSm}_{\bb F_p}\to  D(\bb Z)$ is left Kan extended from $\op{Sm}_{\bb F_p}$.
\item For any $A\in \op{CSm}_{\bb F_p}$, there are natural equivalences \[\bb Z/p^r\bb Z(j)(A)[j]\simeq R\Gamma_\sub{\'et}(\Spec A,W_r\Omega_\sub{log}^j)\simeq[W_r\Omega^j_{A}\xto{\res F-1}W_r\Omega^j_A/dV^{r-1}\Omega^{j-1}_A].\] 
\end{enumerate}
\end{proposition}
\begin{proof}
The second equivalence in (ii) actually holds for any $\bb F_p$-algebra, since on any $\bb F_p$-scheme $X$ there is a short exact sequence of \'etale sheaves $0\to W_r\Omega_{X,\sub{log}}^j\to W_r\Omega^j_{X}\xto{\res F-1}W_r\Omega^j_X/dV^{r-1}\Omega^{j-1}_X\to 0$ where the middle and final term have no higher cohomology on affines \cite[Corol.~4.1(ii) \& 4.2(iii)]{Morrow_pro_GL2}.

For any $A\in \op{CSm}_{\bb F_p}$, there is a natural short exact sequence $0\to W_{r-1}\Omega^j_{A,\sub{log}}\xto{p}W_{r}\Omega^j_{A,\sub{log}}\to \Omega^j_{A,\sub{log}}\to 0$ \cite[Thm. 2.11]{KellyMorrow2018a}, whence part (i) reduces to the case $r=1$. Similarly, since this short exact sequence equally holds for any \'etale $A$-algebra (since it is also Cartier smooth), there is an induced short exact sequence of sheaves on the \'etale site of $A$ and a corresponding fibre sequence of cohomology
\[R\Gamma_\sub{\'et}(\Spec A,W_{r-1}\Omega_\sub{log}^j)\xto{p}R\Gamma_\sub{\'et}(\Spec A,W_r\Omega_\sub{log}^j)\to R\Gamma_\sub{\'et}(\Spec A,\Omega_\sub{log}^j).\] So the claim ``$R\Gamma_\sub{\'et}(\Spec -,W_{r}\Omega_\sub{log}^j):\op{CSm}_{\bb F_p}\to  D(\bb Z)$ is left Kan extended from $\op{Sm}_{\bb F_p}^\Sigma$'' also reduces to the case $r=1$. But this claim would imply the first equivalence in (ii): it is even equivalent to it as the first equivalence in (ii) is already known for all finitely generated polynomial $\bb F_p$-algebras (even all smooth $\bb F_p$-algebras) \cite[Corol.~8.19]{BhattMorrowScholze2} and $\bb Z/p^r\bb Z(j)(-)$ is left Kan extended from $\op{Sm}_{\bb F_p}^\Sigma$ (either by the formula $\bb Z/p^j(j)(-)=\op{hofib}(\tfrac\phi{p^j}-1:\cal N^{\ge j}LW\Omega\to LW\Omega_A)/p^r$ of \cite{BhattMorrowScholze2}, or just by quoting \cite[Thm.~5.1(ii)]{AntieauMathewMorrowNikolaus}).

We have reduced (i) and (ii) to proving that on the category $\op{CSm}_{\bb F_p}$, the functors $\Omega^j_\sub{log}$ and $R\Gamma_\sub{\'et}(\Spec-,\Omega^j_\sub{log})$ are left Kan extended from $\op{Sm}_{\bb F_p}$ and from $\op{Sm}_{\bb F_p}^\Sigma$ respectively.

But it was shown in \cite{KellyMorrow2018a} that both $\Omega^j$ and $d\Omega^{j-1}$, on the category $\op{CSm}_{\bb F_p}$, are left Kan extended from $\op{Sm}_{\bb F_p}^\Sigma$, so the same is true of $R\Gamma_\sub{\'et}(\Spec-,\Omega^j_\sub{log})$ using the second equivalence of part (ii). This equivalence also shows that there is a natural fibre sequence $\Omega^j_\sub{log}\to R\Gamma_\sub{\'et}(\Spec A,\Omega^j_\sub{log})[j]\to \tilde\nu(j)(A)[-1]$ for all $A\in\op{CSm}_{\bb F_p}$; the final term is rigid \cite[Prop.~4.30]{ClausenMathewMorrow2020} hence left Kan extended from $\op{Sm}_{\bb F_p}$, and so we deduce the same for the first term.
\end{proof}

This allows us to complement the main theorem of \cite{KellyMorrow2018a}, namely $K_j(A;\bb Z/p^r\bb Z)\cong W_r\Omega_{A,\sub{log}}^j$, by also describing Milnor $K$-theory:

\begin{theorem}[Bloch--Kato--Gabber theorem for Cartier smooth algebras]\label{theorem_BKG}
For any local, Cartier smooth $\bb F_p$-algebra $A$, the symbol map $\hat K_j^M(A)/p^r\to W_r\Omega_{A,\sub{log}}^j$ is an isomorphism for any $r,j\ge0$.
\end{theorem}
\begin{proof}
Although this follows from Theorem \ref{theorem_Nesterenko_Suslin} and Proposition \ref{proposition_LKE_Omegalog2}(ii), we prefer to give a proof which avoids the passage through mixed characteristic.

The symbol map is surjective by Remark \ref{remark_WrOmegalog} so it remains to treat its injectivity, for which we may assume $A$ has big residue field (otherwise take a finite \'etale extension $A'\supseteq A$ of degree coprime to $p$, with $A'$ also local, so that $\hat K_j^M(A)/p^r\to K_j^M(A')/p^r$ is injective by existence of the norm map). We view $K_j^M(-)/p^r\to W_r\Omega^j_\sub{log}$ as a map of functors on $\op{SCm}_{\bb F_p}^\sub{loc}$. It is an isomorphism on any essentially smooth, local $\bb F_p$-algebra with big residue field by the Bloch--Kato--Gabber theorem (see Remark \ref{remark_WrOmegalog} again); by left Kan extending using Propositions \ref{proposition_LKE_Milnor} and \ref{proposition_LKE_Omegalog2} we obtain the desired isomorphism for $A$. Note here that, in the diagram of essentially smooth, local $\bb F_p$-algebras mapping to $A$, those with big residue field form a cofinal system (alternatively, computing the left Kan extension by the type of simplicial resolution which appeared in Proposition \ref{proposition_LKE_Milnor}, just note that each $R_q$ has the same big residue field as $A$).
\end{proof}

\subsection{A pro Bloch--Kato--Gabber theorem}
We prove a version of the Bloch--Kato--Gabber theorem for pro rings, and an associated continuity theorem for Milnor $K$-theory. The initial two lemmas are very much in the style of \cite{Morrow_pro_GL2}, to which we refer for further background on arguments with pro abelian groups (which we denote by curly brackets $\{A_s\}_s$, rather than $\projlimf_sA_s$) and the role of F-finiteness.

\begin{lemma}\label{lemma_pro_Illusie_sequence}
Let $A$ be an F-finite, regular, Noetherian, local $\bb F_p$-algebra, and $I\subseteq A$ an ideal. Then there is a natural short exact sequence of pro abelian groups \[0\to \{W_{r-1}\Omega^j_{A/I^s,\sub{log}}\}_s\xto{p}\{W_{r}\Omega^j_{A/I^s,\sub{log}}\}_s\to \{\Omega^j_{A/I^s,\sub{log}}\}_s\to 0\] for each $j\ge0$.
\end{lemma}
\begin{proof}
It was shown in \cite[Corol.~4.8]{Morrow_pro_GL2} that the diagonally indexed pro abelian group $\{W_{s}\Omega^j_{A/I^s,\sub{log}}\}_s$ is $p$-torsion-free and the canonical map $\{W_{s}\Omega^j_{A/I^s,\sub{log}}/p^r\}_s\to \{W_{r}\Omega^j_{A/I^s,\sub{log}}\}_s$ is an isomorphism. So the desired short exact sequence is $0\to \{W_{s}\Omega^j_{A/I^s,\sub{log}}/p^{r-1}\}_s\xto{p}\{W_{s}\Omega^j_{A/I^s,\sub{log}}/p^r\}_s\to \{W_{s}\Omega^j_{A/I^s,\sub{log}}/p\}_s\to~0$.
\end{proof}

For a functor $F:\bb F_p\op{-algs}\to\op{Ab}$, let $\bb LF:\bb F_p\op{-algs}\to D^{\le 0}(\bb Z)$ denote the left Kan extension from $\op{Sm}_{\bb F_p}^\Sigma$ of the restriction of $F$. For the application to Theorem \ref{theorem_pro_BKG}, note that for each of the three functors in the lemma it would be equivalent to left Kan extend from $\op{Sm}_{\bb F_p}$ (as follows from the Cartier isomorphism, or alternatively apply the following lemma with $I=0$).

\begin{lemma}\label{lemma_LKE_of_BZ}
Let $A$ be a F-finite, regular, Noetherian $\bb F_p$-algebra and $I\subseteq A$ an ideal. Then the canonical maps of complexes  \[\bb L\Omega^j_{A/I^s}\To \Omega^j_{A/I^s}\qquad \bb LB\Omega^j_{A/I^s}\To B\Omega^j_{A/I^s}\qquad \bb LZ\Omega^j_{A/I^s}\To Z\Omega^j_{A/I^s}\] become equivalences of pro complexes after applying $\{\,\}_s$.
\end{lemma}
\begin{proof}
By ``equivalences of pro complexes'' we simply mean in each case that $\{H_n(-)\}_s=0$ for $n>0$, and that the pro abelian group $\{H_0(-)\}_s$ identifies with target pro abelian group.

Since $H^0(\bb L\Omega^j_{A/I^s})=\Omega^j_{A/I^s}$ for each $s$, in the first case it remains to prove that $\{H_n(\bb L\Omega^j_{A/I^s})\}=0$ for $n>0$. From the transitivity sequence for $\bb F_p\to A\to A/I^s$ one knows that $\bb L\Omega^j_{A/I^s}$ has a natural filtration with graded pieces $\bb L\Omega^i_A\otimes_A\bb L\Omega^{j-i}_{(A/I^s)/A}$ for $i=1,\dots,j$; but a classical argument of M.~Andr\'e \cite[Prop.~X.12]{Andre1974} (see also \cite[Thm.~4.4(i)]{Morrow_pro_H_unitality}) shows that $\{H_n(\bb L\Omega^{j-i}_{(A/I^s)/A})\}_s=0$ unless $n=j-i=0$. Since also $\bb L\Omega^i_A\simeq\Omega_A^i$ as $A$ is geometrically regular over $\bb F_p$, the spectral sequence associated to the filtration (i.e., the spectral sequence of Kassel--Sletsj\o e \cite{Kassel1992}) now yields the desired vanishing.

The pro equivalences for $B\Omega^j$ and $Z\Omega^j$ now follow by the usual increasing induction on $j$, just as in the Cartier smooth case (see the first three paragraphs of the proof of \cite[Prop.~2.5]{KellyMorrow2018a}), since the inverse Cartier map $C^{-1}:\{\Omega^j_{A/I^s}\}_s\to\{H^j(\Omega^\blob_{A/I^s})\}_s$ is known to be an isomorphism of pro abelian groups for all $j\ge0$ \cite[Thm.~2.2]{Morrow_pro_GL2}.
\end{proof}

We may now present the main theorem of the subsection, describing the pro system of Milnor $K$-groups of the successive quotients of a regular, local $\bb F_p$-algebra:

\begin{theorem}\label{theorem_pro_BKG}
Let $A$ be an F-finite, regular, Noetherian, local $\bb F_p$-algebra, and $I\subseteq A$ any ideal. Then the symbol map induces an isomorphism of pro abelian groups $\{\hat K_j^M(A/I^s)/p^r\}_s\isoto\{W_r\Omega^j_{A/I^s,\sub{log}}\}_s$ for any $r,j\ge0$.
\end{theorem}
\begin{proof}
Morally one would like to say that this theorem is a special case of Theorem \ref{theorem_BKG}, as the results of \cite[\S2]{Morrow_pro_GL2} implicitly show that the pro ring $\{A/I^s\}_s$ satisfies a pro analogue of the definition of Cartier smoothness. But we provide the details of the proof.

The symbol map is automatically surjective (even for each fixed level $s$), and for injectivity we perform the same trick as in Theorem \ref{theorem_BKG} to henceforth assume that $A$ has big residue field. For any essentially smooth, local $\bb F_p$-algebra $R$ with big residue field, the symbol map $K_j^M(R)/p^r\to W_r\Omega^j_{R,\sub{log}}$ is an isomorphism as recalled in Remark \ref{remark_WrOmegalog}. Left Kan extending from such $R$ and appealing to Proposition~\ref{proposition_LKE_Milnor}, the theorem reduces to checking that the co-unit map $H^0((L^\sub{sm}W_r\Omega^j_{\sub{log}})(A/I^s))\to W_r\Omega^j_{A/I^s,\sub{log}}$ induces an isomorphism of pro abelian groups over $s$.

This in turn reduces to the case $r=1$ by comparing the fibre sequences \[(L^\sub{sm}W_{r-1}\Omega^j_{\sub{log}})(A/I^s)\xto{p}(L^\sub{sm}W_r\Omega^j_{\sub{log}})(A/I^s)\to (L^\sub{sm}\Omega^j_{\sub{log}})(A/I^s)\] (induced by Illusie's short exact sequence $0\to W_{r-1}\Omega^j_{R,\sub{log}}\xto{p}W_{r}\Omega^j_{R,\sub{log}}\to \Omega^j_{R,\sub{log}}\to 0$ for each essentially smooth, local $\bb F_p$-algebra \cite[\S I.5.7]{Illusie1979}) to the short exact sequence of Lemma \ref{lemma_pro_Illusie_sequence}. We then consider the fibre sequence $\Omega^j_{-,\sub{log}}\to[\Omega_-^j\xto{C^{-1}-1}\Omega^j_-/d\Omega^{j-1}_-]\to\tilde\nu(j)(-)[-1]$ on arbitrary $\bb F_p$-algebras: left Kan extending it from essentially smooth, local $\bb F_p$-algebras, evaluating on $A/I^s$, comparing to the original sequence for $A/I^s$ itself, and using that $\tilde\nu(j)(-)$ is left Kan extended from smooth algebras, the problem is reduced to checking that the co-unit maps $L^\sub{sm}\Omega^j_{A/I^s}\to \Omega^j_{A/I^s}[0]$ and  $L^\sub{sm}B\Omega^j_{A/I^s}\to B\Omega^j_{A/I^s}[0]$ induce equivalences of pro complexes. But this is exactly what we checked in Lemma \ref{lemma_LKE_of_BZ}.
\end{proof}

\begin{corollary}
Let $A$ be an F-finite, regular, Noetherian, local $\bb F_p$-algebra, and $I\subseteq A$ an ideal such that $A$ is $I$-adically complete. Then the canonical map $K_j^M(A)/p^r\to\projlim_sK_j^M(A/I^s)/p^r$ is an isomorphism for any $r\ge1$.
\end{corollary}
\begin{proof}
In light of the classical Bloch--Kato--Gabber theorem and Theorem \ref{theorem_pro_BKG}, it is equivalent to prove that the map $W_r\Omega^j_{A,\sub{log}}\to\projlim_sW_r\Omega^j_{A/I^s,\sub{log}}$ is an isomorphism. This was proved in \cite[Corol.~4.11]{Morrow_pro_GL2}.
\end{proof}

\begin{remark}
Note that, combined with \cite[Thm.~5.27]{ClausenMathewMorrow2020}, Theorem \ref{theorem_pro_BKG} shows that the canonical map $\{\hat K_j^M(A/I^s)/p^r\}_s\isoto  \{K_j(A/I^s)/p^r\}_s$ is an isomorphism.

Moreover, just as explained in \cite[Rem.~5.8]{Morrow_pro_GL2}, the implicit bounds appearing in the isomorphisms of pro systems are uniform when localising (or more generally passing to \'etale algebras). Therefore, for any F-finite, regular, Noetherian $\bb F_p$-scheme $X$, and closed subscheme $Y\into X$, the canonical maps of (Zariski, Nisnevich, or \'etale) sheaves on $Y$ \[\cal K_{j,Y_s}/p^r\longleftarrow\hat{\cal K}_{j,Y_s}^M/p^s\xto{\dlog}W_r\Omega^j_{Y_s,\sub{log}}\] becomes an isomorphism of pro sheaves after applying $\{\}_s$.
\end{remark}

\section*{Appendix: A Gersten injectivity result of Gabber}
The following injectivity result of Gabber was required in the proof of Theorem \ref{theorem_BK_unram}: 

\begin{theorem}[Gabber]\label{theorem_Gabber_injectivity}
Let $V$ be a mixed characteristic discrete valuation ring and $R$ a local, ind-smooth $p$-henselian $V$-algebra; let $R_{\frak pR}^h$ be the henselisation of the discrete valuation ring $R_{\frak pR}$, and $F:=R_{\frak pR}^h[\tfrac1p]$ its field of fractions. Let $\cal F$ be a torsion \'etale sheaf on $\Spec R[\tfrac1p]$ which is pulled back from $\Spec V[\tfrac1p]$. Then the canonical map \[H^n_\sub{\'et}(R[\tfrac1p],\cal F)\To H^n_\sub{\'et}(F,\cal F)\] is injective for all $n\ge0$.
\end{theorem}
\begin{proof}
Since both sides commute with filtered colimits in $R$, we may assume that $R$ is the henselisation along $\frak pS_\frak q$ of $S_\frak q$, where $S$ is a smooth $V$-algebra and $\frak q\subseteq S$ is some prime ideal containing $\frak pS$; note that then $R_{\frak pR}^h=S_{\frak pS}^h$, where the latter denotes the henselisation of the discrete valuation ring $S_{\frak pS}$.

The beginning of Gabber's Gersten resolution \cite[Eqn.~($\ast$)]{Gabber1994}, at the point $\frak q$ of $\Spec S$ (Gabber's scheme $M$), asserts that the map \[H^n_\sub{\'et}(\Spec S_\frak q/\frak p S_\frak q,i^*Rj_*\cal F)\To H^n_\sub{\'et}(\Spec S_{\frak pA}/\frak p S_{\frak pA},i^*Rj_*\cal F)\] is injective, where $i,j$ are the usual closed and open inclusions $\Spec S/\frak pS\stackrel{i}{\into}\Spec S\xleftarrow{j}\Spec S[\tfrac1p]$, and we suppress the additional pullbacks along $\Spec S_{\frak pR}/\frak p S_{\frak pS}\to \Spec S_\frak q/\frak p S_\frak q\to\Spec S/\frak pS$ from the notation.

Noting that $R/\frak pR=S_\frak q/\frak pS_\frak q$, Gabber's affine analogue of the proper base change theorem \cite{} implies that the canonical map \[H^n_\sub{\'et}(\Spec R[\tfrac1p],\cal F)=H^n_\sub{\'et}(\Spec R,Rj_*\cal F)\To H^n_\sub{\'et}(\Spec S_\frak q/\frak pS_\frak q,i^*Rj_*\cal F)\] is an isomorphism. The analogous assertion is equally true for the henselian surjection $R_{\frak pR}^h\to R_{\frak pA}/\frak pR_{\frak pA}=S_{\frak pA}/\frak pS_{\frak pA}$, thereby completing the proof.
\end{proof}

\def\cprime{$'$}

\noindent
\parbox{0.5\linewidth}{
\noindent
Morten L\"uders \\IMJ-PRG,\\
SU -- 4 place Jussieu,\\
Case 247,\\
75252 Paris\\
{\tt morten.luders@imj-prg.fr}
}
\parbox{0.4\linewidth}{
Matthew Morrow\\
CNRS \& IMJ-PRG,\\
SU -- 4 place Jussieu,\\
Case 247,\\
75252 Paris\\
{\tt matthew.morrow@imj-prg.fr}
}


\begin{thebibliography}{10}

\bibitem{Andre1974}
{\sc Andr{\'e}, M.}
\newblock {\em Homologie des alg\`ebres commutatives}.
\newblock Springer-Verlag, Berlin, 1974.
\newblock Die Grundlehren der mathematischen Wissenschaften, Band 206.

\bibitem{AntieauMathewMorrowNikolaus}
{\sc Antieau, B., Mathew, A., Morrow, M., and Nikolaus, T.}
\newblock On the {B}eilinson fiber square.
\newblock {\em {\tt arXiv:2003.12541}\/} (2020).

\bibitem{BhattMorrowScholze2}
{\sc Bhatt, B., Morrow, M., and Scholze, P.}
\newblock Topological {H}ochschild homology and integral {$p$}-adic {H}odge
  theory.
\newblock {\em Publ. Math. Inst. Hautes \'{E}tudes Sci. 129\/} (2019),
  199--310.

\bibitem{Bloch1983}
{\sc Bloch, S.}
\newblock {$p$}-adic \'etale cohomology.
\newblock In {\em Arithmetic and geometry, {V}ol. {I}}, vol.~35 of {\em Progr.
  Math.} Birkh\"auser Boston, Boston, MA, 1983, pp.~13--26.

\bibitem{Bloch1986c}
{\sc Bloch, S.}
\newblock A note on {G}ersten's conjecture in the mixed characteristic case.
\newblock In {\em Applications of algebraic {$K$}-theory to algebraic geometry
  and number theory, {P}art {I}, {II} ({B}oulder, {C}olo., 1983)}, vol.~55 of
  {\em Contemp. Math.} Amer. Math. Soc., Providence, RI, 1986, pp.~75--78.

\bibitem{Bloch1986}
{\sc Bloch, S., and Kato, K.}
\newblock {$p$}-adic \'etale cohomology.
\newblock {\em Inst. Hautes \'Etudes Sci. Publ. Math.}, 63 (1986), 107--152.

\bibitem{ClausenMathewMorrow2020}
{\sc Clausen, D., Mathew, A., and Morrow, M.}
\newblock {$K$}-theory and topological cyclic homology of henselian pairs.
\newblock {\em Journals of the AMS, to appear\/} (2018).

\bibitem{ColliotThelene1999}
{\sc Colliot-Th\'{e}l\`ene, J.-L.}
\newblock Cohomologie galoisienne des corps valu\'{e}s discrets henseliens,
  d'apr\`es {K}. {K}ato et {S}. {B}loch.
\newblock In {\em Algebraic {$K$}-theory and its applications ({T}rieste,
  1997)}. World Sci. Publ., River Edge, NJ, 1999, pp.~120--163.

\bibitem{Dahlhausen2018}
{\sc Dahlhausen, C.}
\newblock Milnor {K}-theory of complete discrete valuation rings with finite
  residue fields.
\newblock {\em J. Pure Appl. Algebra 222}, 6 (2018), 1355--1371.

\bibitem{Dennis1973a}
{\sc Dennis, R.~K., and Stein, M.~R.}
\newblock {$K_{2}$} of radical ideals and semi-local rings revisited.
\newblock In {\em Algebraic {$K$}-theory, {II}: ``{C}lassical'' algebraic
  {$K$}-theory and connections with arithmetic ({P}roc. {C}onf., {B}attelle
  {M}emorial {I}nst., {S}eattle, {W}ash., 1972)}. Springer, Berlin, 1973,
  pp.~281--303. Lecture Notes in Math. Vol. 342.

\bibitem{Dennis1975}
{\sc Dennis, R.~K., and Stein, M.~R.}
\newblock {$K\sb{2}$} of discrete valuation rings.
\newblock {\em Advances in Math. 18}, 2 (1975), 182--238.

\bibitem{ElmantoHoyoisKhanSosniloYakerson2019}
{\sc Elmanto, E., Hoyois, M., Khan, A.~A., Sosnilo, V., and Yakerson, M.}
\newblock Modules over algebraic cobordism.
\newblock {\em Forum.\ Math.\ Pi, to appear\/} (2019).

\bibitem{Fesenko2002}
{\sc Fesenko, I.~B., and Vostokov, S.~V.}
\newblock {\em Local fields and their extensions}, second~ed., vol.~121 of {\em
  Translations of Mathematical Monographs}.
\newblock American Mathematical Society, Providence, RI, 2002.
\newblock With a foreword by I. R. Shafarevich.

\bibitem{Gabber1994}
{\sc Gabber, O.}
\newblock Gersten's conjecture for some complexes of vanishing cycles.
\newblock {\em Manuscripta Math. 85}, 3-4 (1994), 323--343.

\bibitem{GabberRamero2003}
{\sc Gabber, O., and Ramero, L.}
\newblock {\em Almost ring theory}, vol.~1800 of {\em Lecture Notes in
  Mathematics}.
\newblock Springer-Verlag, Berlin, 2003.

\bibitem{Geisser2004}
{\sc Geisser, T.}
\newblock Motivic cohomology over {D}edekind rings.
\newblock {\em Math. Z. 248}, 4 (2004), 773--794.

\bibitem{GeisserLevine2000}
{\sc Geisser, T., and Levine, M.}
\newblock The {$K$}-theory of fields in characteristic {$p$}.
\newblock {\em Invent. Math. 139}, 3 (2000), 459--493.

\bibitem{Gille2006}
{\sc Gille, P., and Szamuely, T.}
\newblock {\em Central simple algebras and {G}alois cohomology}, vol.~101 of
  {\em Cambridge Studies in Advanced Mathematics}.
\newblock Cambridge University Press, Cambridge, 2006.

\bibitem{GilletLevine1987}
{\sc Gillet, H., and Levine, M.}
\newblock The relative form of {G}ersten's conjecture over a discrete valuation
  ring: the smooth case.
\newblock {\em J. Pure Appl. Algebra 46}, 1 (1987), 59--71.

\bibitem{GrosSuwa1988}
{\sc Gros, M., and Suwa, N.}
\newblock Application d'{A}bel-{J}acobi {$p$}-adique et cycles alg\'ebriques.
\newblock {\em Duke Math. J. 57}, 2 (1988), 579--613.

\bibitem{Illusie1979}
{\sc Illusie, L.}
\newblock Complexe de de\thinspace {R}ham-{W}itt et cohomologie cristalline.
\newblock {\em Ann. Sci. \'Ecole Norm. Sup. (4) 12}, 4 (1979), 501--661.

\bibitem{Izhboldin1991}
{\sc Izhboldin, O.}
\newblock On {$p$}-torsion in {$K^M_*$} for fields of characteristic {$p$}.
\newblock In {\em Algebraic {$K$}-theory}, vol.~4 of {\em Adv. Soviet Math.}
  Amer. Math. Soc., Providence, RI, 1991, pp.~129--144.

\bibitem{Izhboldin2000}
{\sc Izhboldin, O.}
\newblock {$p$}-primary part of the {M}ilnor {$K$}-groups and {G}alois
  cohomologies of fields of characteristic {$p$}.
\newblock In {\em Invitation to higher local fields ({M}\"{u}nster, 1999)},
  vol.~3 of {\em Geom. Topol. Monogr.} Geom. Topol. Publ., Coventry, 2000,
  pp.~19--41.
\newblock With an appendix by Masato Kurihara and Ivan Fesenko.

\bibitem{Kassel1992}
{\sc Kassel, C., and Sletsj{\o}e, A.~B.}
\newblock Base change, transitivity and {K}\"unneth formulas for the {Q}uillen
  decomposition of {H}ochschild homology.
\newblock {\em Math. Scand. 70}, 2 (1992), 186--192.

\bibitem{KellyMorrow2018a}
{\sc Kelly, S., and Morrow, M.}
\newblock {$K$}-theory of valuation rings.
\newblock {\em Compositio Mathematica, to appear\/} (2018).

\bibitem{Kerz2009}
{\sc Kerz, M.}
\newblock The {G}ersten conjecture for {M}ilnor {$K$}-theory.
\newblock {\em Invent. Math. 175}, 1 (2009), 1--33.

\bibitem{Kerz2010}
{\sc Kerz, M.}
\newblock Milnor {$K$}-theory of local rings with finite residue fields.
\newblock {\em J. Algebraic Geom. 19}, 1 (2010), 173--191.

\bibitem{EsnaultKerzWittenberg2016}
{\sc Kerz, M., Esnault, H., and Wittenberg, O.}
\newblock A restriction isomorphism for cycles of relative dimension zero.
\newblock {\em Camb. J. Math. 4}, 2 (2016), 163--196.

\bibitem{KerzStrunkTamme2018}
{\sc {Kerz}, M., {Strunk}, F., and {Tamme}, G.}
\newblock {Algebraic \(K\)-theory and descent for blow-ups.}
\newblock {\em {Invent. Math.} 211}, 2 (2018), 523--577.

\bibitem{Kolster1985}
{\sc Kolster, M.}
\newblock {$K_2$} of noncommutative local rings.
\newblock {\em J. Algebra 95}, 1 (1985), 173--200.

\bibitem{Kurihara1988}
{\sc Kurihara, M.}
\newblock Abelian extensions of an absolutely unramified local field with
  general residue field.
\newblock {\em Invent. Math. 93}, 2 (1988), 451--480.

\bibitem{Lueders2019}
{\sc L\"{u}ders, M.}
\newblock Algebraization for zero-cycles and the {$p$}-adic cycle class map.
\newblock {\em Math. Res. Lett. 26}, 2 (2019), 557--585.

\bibitem{Lueders2020}
{\sc L\"{u}ders, M.}
\newblock Deformation theory of the {C}how group of zero cycles.
\newblock {\em Q. J. Math. 71}, 2 (2020), 677--701.

\bibitem{Lueders2020a}
{\sc L\"{u}ders, M.}
\newblock On the relative {G}ersten conjecture for {M}ilnor {$K$}-theory in the
  smooth case.
\newblock {\em {\tt arXiv:2010.02622}\/} (2020).

\bibitem{Morrow_K2}
{\sc Morrow, M.}
\newblock {$K_2$} of localisations of local rings.
\newblock {\em Journal of Algebra 399\/} (2014), 190--204.

\bibitem{Morrow_pro_H_unitality}
{\sc Morrow, M.}
\newblock Pro unitality and pro excision in algebraic {$K$}-theory and cyclic
  homology.
\newblock {\em J. Reine Angew. Math. 736\/} (2018), 95--139.

\bibitem{Morrow_pro_GL2}
{\sc {Morrow}, M.}
\newblock {\(K\)-theory and logarithmic Hodge-Witt sheaves of formal schemes in
  characteristic \(p\)}.
\newblock {\em {Ann. Sci. \'Ec. Norm. Sup\'er. (4)} 52}, 6 (2019), 1537--1601.

\bibitem{Nakamura2000a}
{\sc Nakamura, J.}
\newblock On the {M}ilnor {$K$}-groups of complete discrete valuation fields.
\newblock {\em Doc. Math. 5\/} (2000), 151--200.

\bibitem{Nakamura2000}
{\sc Nakamura, J.}
\newblock On the structure of the {M}ilnor {$K$}-groups of complete discrete
  valuation fields.
\newblock In {\em Invitation to higher local fields ({M}\"{u}nster, 1999)},
  vol.~3 of {\em Geom. Topol. Monogr.} Geom. Topol. Publ., Coventry, 2000,
  pp.~123--135.

\bibitem{Suslin1989}
{\sc Nesterenko, Y.~P., and Suslin, A.~A.}
\newblock Homology of the general linear group over a local ring, and
  {M}ilnor's {$K$}-theory.
\newblock {\em Izv. Akad. Nauk SSSR Ser. Mat. 53}, 1 (1989), 121--146.

\bibitem{Nisnevich1989}
{\sc Nisnevich, Y.~A.}
\newblock The completely decomposed topology on schemes and associated descent
  spectral sequences in algebraic {$K$}-theory.
\newblock In {\em Algebraic {$K$}-theory: connections with geometry and
  topology ({L}ake {L}ouise, {AB}, 1987)}, vol.~279 of {\em NATO Adv. Sci.
  Inst. Ser. C Math. Phys. Sci.} Kluwer Acad. Publ., Dordrecht, 1989,
  pp.~241--342.

\bibitem{Panin2003}
{\sc Panin, I.~A.}
\newblock The equicharacteristic case of the {G}ersten conjecture.
\newblock {\em Tr. Mat. Inst. Steklova 241}, Teor. Chisel, Algebra i Algebr.
  Geom. (2003), 169--178.

\bibitem{Quillen1973a}
{\sc {Quillen}, D.}
\newblock {Higher algebraic K-theory. I.}
\newblock {Algebr. K-Theory I, Proc. Conf. Battelle Inst. 1972, Lect. Notes
  Math. 341, 85-147 (1973).}, 1973.

\bibitem{Sato2007}
{\sc Sato, K.}
\newblock {$p$}-adic \'etale {T}ate twists and arithmetic duality.
\newblock {\em Ann. Sci. \'Ecole Norm. Sup. (4) 40}, 4 (2007), 519--588.
\newblock With an appendix by Kei Hagihara.

\bibitem{Schneider1994}
{\sc Schneider, P.}
\newblock {$p$}-adic points of motives.
\newblock In {\em Motives ({S}eattle, {WA}, 1991)}, vol.~55 of {\em Proc.
  Sympos. Pure Math.} Amer. Math. Soc., Providence, RI, 1994, pp.~225--249.

\bibitem{Suslin1991}
{\sc Suslin, A.~A., and Yarosh, V.~A.}
\newblock Milnor's {$K_3$} of a discrete valuation ring.
\newblock In {\em Algebraic {$K$}-theory}, vol.~4 of {\em Adv. Soviet Math.}
  Amer. Math. Soc., Providence, RI, 1991, pp.~155--170.

\bibitem{Tate1976}
{\sc Tate, J.}
\newblock Relations between {$K_{2}$} and {G}alois cohomology.
\newblock {\em Invent. Math. 36\/} (1976), 257--274.

\bibitem{Totaro1992}
{\sc Totaro, B.}
\newblock Milnor {$K$}-theory is the simplest part of algebraic {$K$}-theory.
\newblock {\em $K$-Theory 6}, 2 (1992), 177--189.

\bibitem{vanderKallen1977}
{\sc van~der Kallen, W.}
\newblock The {$K\sb{2}$} of rings with many units.
\newblock {\em Ann. Sci. \'Ecole Norm. Sup. (4) 10}, 4 (1977), 473--515.

\bibitem{vanderKallen1975}
{\sc van~der Kallen, W., Maazen, H., and Stienstra, J.}
\newblock A presentation for some {$K\sb{2}(n,R)$}.
\newblock {\em Bull. Amer. Math. Soc. 81}, 5 (1975), 934--936.

\bibitem{Weibel2013}
{\sc Weibel, C.~A.}
\newblock {\em The {$K$}-book}, vol.~145 of {\em Graduate Studies in
  Mathematics}.
\newblock American Mathematical Society, Providence, RI, 2013.
\newblock An introduction to algebraic $K$-theory.

\end{thebibliography}
\end{document}